\documentclass[final,leqno,letterpaper]{etna}

\setbibdata{1}{xx}{yy}{202z} %First page/Last page/volume/year

\hypersetup{%
    pdftitle={Symmetrization Techniques in Image Deblurring},
    pdfauthor={Marco Donatelli, Paola Ferrari, and Silvia Gazzola},
    pdfkeywords={a keyword, another keyword, and another one}
    }

\title{Symmetrization Techniques in Image Deblurring\thanks{%
Received... Accepted... Published online on... Recommended by....
}}

\author{Marco Donatelli\footnotemark[2]
        \and Paola Ferrari\footnotemark[2]
        \and Silvia Gazzola\footnotemark[3]}

\shorttitle{Symmetrization Techniques in Image Deblurring} 
\shortauthor{M.~DONATELLI AND P.~FERRARI AND S.~GAZZOLA}

\usepackage{bm}
\usepackage{mathdots}
\usepackage{soul, color}
\usepackage{amsmath}
\usepackage{amssymb} 
\usepackage{mathrsfs}
\usepackage{graphicx}
\usepackage{epstopdf}
\usepackage{epsfig}
\usepackage{yfonts}
\usepackage{subcaption}
\usepackage{todonotes}

\newcommand{\R}{{\mathbb R}}

\newcommand{\diff}{\rm{d}}

\newcommand{\bs}{{\mbox{\boldmath{$s$}}}}

\newcommand{\bxi}{{\mbox{\boldmath{$\xi$}}}}
\newcommand{\bb}{{\mbox{\boldmath{$b$}}}}

\newcommand{\bt}{{\mbox{\boldmath{$t$}}}}
\newcommand{\bx}{{\mbox{\boldmath{$x$}}}}
\renewcommand{\bb}{{\mbox{\boldmath{$b$}}}}
\newcommand{\by}{{\mbox{\boldmath{$y$}}}}
\newcommand{\bz}{{\mbox{\boldmath{$z$}}}}
\newcommand{\imm}{\mathrm{\hat{\imath}}}
\newcommand{\toep}{\mathcal{T}}
\newcommand{\cir}{\mathcal{C}}
\newcommand{\e}{\mathrm{e}}
\newcommand{\IRTools}{\textit{IR Tools}}

\begin{document}

\maketitle

\begin{center} {\em Dedicated to Lothar Reichel on his seventieth birthday}
\end{center}

\renewcommand{\thefootnote}{\fnsymbol{footnote}}

\footnotetext[2]{Dipartimento di Scienza e Alta Tecnologia, Università dell’Insubria, via Valleggio 11, 22100 Como, Italy
(\texttt{paola.ferrari@uninsubria.it, marco.donatelli@uninsubria.it}).}
\footnotetext[3]{Department of Mathematical Sciences, University of Bath, Bath, UK (\texttt{S.Gazzola@bath.ac.uk}).}

\begin{abstract}
This paper presents a couple of preconditioning techniques that can be used to enhance the performance of iterative regularization methods applied to image deblurring problems with a variety of point spread functions (PSFs) and boundary conditions. More precisely, we first consider the anti-identity preconditioner, which symmetrizes the coefficient matrix associated to problems with zero boundary conditions, allowing the use of MINRES as a regularization method. When considering more sophisticated boundary conditions and strongly nonsymmetric PSFs, the anti-identity preconditioner improves the performance of GMRES. We then consider both stationary and iteration-dependent regularizing circulant preconditioners that, applied in connection with the anti-identity matrix and both standard and flexible Krylov subspaces, speed up the iterations. A theoretical result about the clustering of the eigenvalues of the preconditioned matrices is proved in a special case. The results of many numerical experiments are reported to show the effectiveness of the new preconditiong techniques, including when considering the deblurring of sparse images. 
\end{abstract}



%\tableofcontents

\section{Introduction}\label{sec:intro}
We consider the restoration of blurred and noise-corrupted images in two space-dimensions. 
By assuming that the point spread function (PSF) is spatially-invariant, the blurring is modeled as a convolution of the form
\begin{equation}\label{eq:conv}
b(\bs)=[\mathcal{K}f](\bs)+\xi(\bs)=\int_{\R^2}h(\bs-\bt)x(\bt)\diff\bt+\xi(\bs), 
\qquad \bs \in \Omega \subset \R^2,
\end{equation}
where $b$ represents the observed (blurred and noisy) 
image, $x$ the (unknown) exact image, $h$ the PSF  with compact support, and $\xi$ the random noise.
The real-valued nonnegative functions $x$ and $b$ determine the light intensity of the desired and available images, respectively.
We assume that the PSF $h$, and thus the blurring phenomenon, is known.
%Therefore, our goal is the computation of an approximation of $f$ given the available  $g$ and the PSF $h$.

Discretization of the above integral equation at equidistant nodes yields
\begin{equation}\label{underdet}
{b}_i=\sum_{j\in\mathbb{Z}^2}h_{i-j}x_j+\xi_i, \qquad i\in\mathbb{Z}^2,
\end{equation}
where the entries of the discrete images $\bb=[b_i]$ and $\bx=[x_j]$ represent the 
light intensity at each pixel and $\bxi=[\xi_i]$ models the noise-contamination at these 
pixels.
Moreover, the observed image is available only in a finite region, the field of view (FOV) corresponding to $i\in[1,n]^2$, 
which is assumed to be square only for notational simplicity. 
Therefore, when there are nonvanishing coefficients $h_i$ 
with $i\ne 0$, the measured intensities near the boundary are affected by
data outside the FOV depending on the support of the PSF. Thus the linear system of equations defined by (\ref{underdet}) 
is  underdetermined, since then there are $n^2$ constraints, while the number of unknowns 
required to specify the equations is larger. 
A meaningful solution of this underdetermined system can be determined in several
ways; see \cite{almeida13,SISC16} for discussions of this approach. Alternatively,
one can determine a linear system of equations with a square matrix,
\begin{equation}\label{original}
A\bx=\bb, \qquad A\in\mathbb{R}^{n^2\times n^2},\qquad \bx,\bb\in\R^{n^2},
\end{equation}
by imposing boundary conditions, where the $x_j$-values in \eqref{underdet} at pixels
outside the FOV are assumed to be certain linear combinations of values inside the FOV; see \cite{IP06,book}. 
%Popular boundary conditions include zero Dirichlet boundary conditions, periodic 
%boundary conditions, reflective boundary conditions, 
%and anti-reflective boundary conditions, see \cite{IP06} for a discussion and comparisons.  

Since the singular values of the discrete convolution operator $A$ gradually approach zero without a significant gap, $A$ is
ill-conditioned and may be numerically rank-deficient; the degree of ill-posedness depends on the decay of the PSF values, see equation \eqref{genfunct} and the analysis in Section \ref{sec:structured_matrices}. A linear system of equations \eqref{original} 
with a matrix of this kind is commonly referred as linear discrete ill-posed problem and requires regularization
 \cite{hansen}.

The structure of the matrix $A$ depends on the boundary conditions. For instance, by using zero boundary conditions we get a Block Toeplitz with Toeplitz Blocks (BTTB) matrix, by using periodic boundary conditions we get a Block Circulant
with Circulant Blocks (BCCB) matrix, while more sophisticated boundary conditions, as reflective, antireflective, and synthetic, give rise to more complex matrix structures \cite{IP06, jim}. 
%and the spectral decomposition of the matrix $A$ cannot always be computed by fast trigonometric transforms \cite{book}. 
Regardless of how complicated the structure of the matrix $A$ is, its matrix-vector product can always be computed in $O(n^2\log(n))$ by padding the vector according to the boundary conditions and then applying the circular convolution by fast Fourier transforms (FFTs), as implemented in the Matlab toolbox \IRTools\ \cite{ir}. Therefore, one generally resorts to iterative methods for the restoration of large images.

The adjoint of the convolution operator in \eqref{eq:conv} is the correlation operator
\begin{equation}\label{eq:corr}
[\mathcal{K}^*x](\bs)=\int_{\R^2}h(\bt-\bs)x(\bt)\diff\bt,
\end{equation}
where we have used the fact that $h$ is real-valued. 
Discretization of \eqref{eq:corr} with the same boundary conditions used for \eqref{original} can be simply obtained from the PSF rotated $180^\circ$. The resulting matrix is denoted as $A'$. Therefore, matrix-vector products with $A'$, i.e., the discretization of the adjoint operator, can be computed rotating the PSF and then applying the same procedure described above for $A$. 
This is the common implementation of the matrix-vector products with the adjoint operator of $A$ when zero or periodic boundary conditions are imposed.
Unfortunately, the matrix $A'$ could differ from $A^T$ when the imposed boundary conditions are not zero Dirichlet or periodic; see \cite{IP06} for details. In such cases, using solvers like CGLS or LSQR with $A^T$ replaced by $A'$ lacks theoretical justification, which makes it natural to explore the performance of other iterative methods that do not require the adjoint operator. Some recent strategies are based on the preconditioned Arnoldi method and nonstationary iterations \cite{SISC16,martin,david,silvia14}.

In this paper we compute a solution of \eqref{original} through iterative regularization methods, which should terminate when a desired approximation is obtained and before noise starts to show up in the solution and thus the restoration error grows (this is the so-called \emph{semi-convergence} phenomenon). For this reason, a reliable stopping criterion is crucial to obtain a good reconstruction. On one hand, 
preconditioning is usually applied to speed up the convergence of iterative methods replacing the linear system \eqref{original} by the following 
\begin{equation}\label{eq:prec}
PA\bx=P\bb \qquad \mbox{ or } \qquad
AP\bz=\bb, \; \bx=P\bz,
\end{equation}
where $P$ is the preconditioner that could be applied to the left or right side of the matrix $A$. 
If iterations are stopped by the statistically-inspired discrepancy principle, right preconditioning is preferred because it does not modify the noise statistics; see \cite{hansen,MR2584074}.
For discrete ill-posed problems, $P$ must be chosen carefully by avoiding clustering of eigenvalues in the so-called noise subspace and exacerbate semi-convergence, since the signal components in this subspace are usually dominated by noise \cite{HNPprec}. 
On the other hand, when the linear systems \eqref{eq:prec} are solved by a Krylov method, the approximate solution is computed in a different subspace and thus the choice of $P$ affects the quality of the restored image rather than speeding the convergence, or possibly it achieves both \cite{HNPprec,claudio,pietro}. In particular, to provide a good restoration, $P$ should symmetrize the operator $A$ and thus $P=A'$ is a favorable choice. This is the so-called reblurring strategy proposed in \cite{IP06} and later studied in \cite{david} in connection to Arnoldi methods. This approach has been further improved adding a clustering of the eigenvalues in the signal space to obtain a fast convergence \cite{pietro,martin}.

Symmetrization of Toeplitz and BTTB linear systems arising from well-posed problems was recently explored by Pestana and Wathen in \cite{doi:10.1137/140974213}. In detail, defining the anti-identity matrix $Y\in \mathbb{R}^{n^2 \times n^2}$ as
	\[
	Y=\begin{bmatrix}{}
	& & 1 \\
	& \iddots & \\
	1 &  & \end{bmatrix},
	\]
the matrix $YA$ is symmetric whenever $A$ is persymmetric, i.e., $YA=A^TY$, as in the case of Toeplitz and BTTB matrices. It follows that the
linear system \eqref{original} can be replaced by the equivalent linear system
\begin{equation}\label{eq:Ysys}
	YA\bx=Y\bb,
\end{equation}
which can be solved by methods that work with symmetric indefinite matrices, such as MINRES and MR-II. When $A$ is a BTTB matrix, as in the case of zero Dirichlet boundary conditions, preconditioning the linear system \eqref{eq:Ysys} by BCCB matrices has been proposed and studied independently in \cite{MR4304085} and \cite{FFHMS} proving the eigenvalue clustering at the two points $-1$ and~$1$. However, such a symmetrization strategy has never been explored for discrete ill-posed problems, where the preconditioner has to deal with the noise subspace.

In this paper, motivated by the importance of having an operator close to symmetric to generate the Krylov subspace in which to search for an approximate solution of a discrete ill-posed problem (see, for instance, \cite{MR2312499}), we investigate the symmetrization technique \eqref{eq:Ysys} for image deblurring problems.  More specifically, the contributions of this article are twofold.
Firstly, we consider zero Dirichlet boundary conditions so that $A$ is a BTTB matrix and we investigate the regularizing properties of %the MR-II method \cite{MR1413298}, which is a variant of 
MINRES, applied to the linear system \eqref{eq:Ysys}. For the symmetrized linear system, we then define a regularizing preconditioner $P$ for the matrix $YA$ combining the analysis in \cite{FFHMS} with the regularizing preconditioner used, for instance, in \cite{pietro,martin}. We prove that the spectrum of the preconditioned matrix $PYA$ is clustered at the three points $\{-1,0,1\}$. Preconditioned MR-II for deblurring astronomical images  has been previously investigated in \cite{HN96}, but using a different symmetrization strategy for PSFs close to symmetric, whereas our approach is also effective for strongly nonsymmetric PSFs such as the motion blur considered in the numerical results. 

The second contribution of this paper is to extend the proposal to generic boundary conditions and more sophisticated regularization methods, such as hybrid projection methods that enforce some sparsity in the computed solution 
\cite{newsurvey}. Since the matrix $A$ might not be persymmetric, $YA$ is not symmetric even though it is close to being symmetric. Thus MINRES is replaced by GMRES \cite{calvetti}. Since the regularizing preconditioner $P$ depends on a parameter, nonstationary preconditioning is explored together with flexible GMRES to avoid the a priori estimation of such parameter, as proposed in \cite{pietrospringer}. When enforcing some sparsity into the computed solution is appropriate, we adopt efficient algorithms for $1$-norm regularization based on an iteratively reweighted least squares, which handle the inverted weights as iteration-dependent preconditioners that modify the approximation subspace within methods based on the flexible Golub-Kahan decomposition (such as FLSQR \cite{MR4331956}) or methods based on the flexible Arnoldi decomposition (such as FGMRES \cite{silvia14}). This approach results in FGMRES and FLSQR methods where two iteration-dependent preconditioners ($P$ and the inverted weights) are sequentially applied at each iteration; to the best of our knowledge, such scheme has never been applied to FLSQR before. 

This paper is organised as follows. Section \ref{sec:structured_matrices} provides some background material on the links between boundary conditions for image deblurring problems and structured matrices appearing in the linear system formulation \eqref{original}, including their associated spectral decompositions. Section \ref{sec:reg_prec} describes the preconditioners considered in this paper and the spectral analysis of the preconditioned matrices. Section \ref{sec:itReg} summarizes the iterative regularization methods considered in this paper, and specifies the strategies adopted to precondition them. Section \ref{sec:numExp} displays the results of four different test problems, which show the performance of the new preconditioners applied in different settings. Section \ref{sec:end} presents some conclusions and outlines some possible extensions to the present work. 
%------------------------------------------------------------------------------------------------------------------------------
\section{Boundary conditions and structured matrices}\label{sec:structured_matrices}
Let $h_{i,j}$ be the entries of the PSF, with $i,j \in \mathbb{Z}$, where $h_{0,0}$ is the central entry. Because of the compact support of the PSF, $h_{i,j}=0$ for $|i|$ or $|j|$ large enough, especially when $h_{i,j}$ is outside the FOV, i.e.,  $\min\{|i|,|j|\} \geq n$. 
Given the coefficients $h_{i,j}$, it is possible to associate to the PSF the so-called generating
function $f: [-\pi,\pi]^{2} \rightarrow \mathbb{C}$ as follows
\begin{equation}\label{genfunct}
f(\vartheta_1,\vartheta_2) = 
\sum_{i,j=-n+1}^{n-1} h_{i,j}\e^{\imm (i \vartheta_{1}+j \vartheta_{2})}, \qquad 
\imm^2 = -1.
\end{equation}
Note that $h_{i,j}$ are the Fourier coefficients of the function $f$.

The structure of the matrix $A\in \mathbb{R}^{n^2 \times n^2}$ depends on the coefficients $h_{i,j}$ and the imposed boundary conditions. 
When the exact image has a black background, as for instance in astronomical imaging, zero boundary conditions are to be preferred. In such case
\[
A=\left(
\begin{array}{cccc}
T_{0} & T_{-1} & \cdots  & T_{-n+1} \\
T_{1} & \ddots  & \ddots  & \vdots  \\
\vdots  & \ddots  & \ddots  & T_{-1} \\
T_{n-1} & \cdots  & T_{1} & T_{0}%
\end{array}%
\right) _{n^{2}\times n^{2}},
\]
with
\[T_{k}=\left(
\begin{array}{cccc}
h_{k,0} & h_{k,-1} & \cdots  & h_{k,-n+1} \\
h_{k,1} & \ddots  & \ddots  & \vdots  \\
\vdots  & \ddots  & \ddots  & h_{k,-1} \\
h_{k,n-1} & \cdots  & h_{k,1} & h_{k,0}%
\end{array}%
\right) _{n\times n}, \qquad k=-n+1, \dots, n-1,
\]
which is a BTTB matrix. In this case we use the notation
\begin{equation}\label{Abttb}
A=\toep_n(f),
\end{equation}
where $f$ is the generating function defined in \eqref{genfunct}. 
If the PSF is not quadrantally symmetric, i.e., symmetric in both horizontal and vertical direction, then $A$ is not symmetric. On the other hand, it is always persymmetric independently of the PSF.

By imposing boundary conditions different from the zero Dirichlet ones, the matrix $A$ is no longer BTTB, but small norm and/or small rank corrections are added to $\toep_n(f)$ depending on the support and the decay rate of the PSF. Matrix-vector products with $A$ can always be computed in $O(n^2\log(n))$ operations by padding the vector according to the boundary conditions and then applying the circular convolution as in \IRTools\ \cite{ir}. 

Among the various kinds of boundary conditions, the periodic ones are computationally attractive, since the resulting matrix $A$ is BCCB and can be diagonalized by FFTs \cite{book}. The BCCB matrix $A$ associated to a PSF, and thus its generating function $f$ defined in \eqref{genfunct}, will be denoted by  
$A=\cir_n(f).$
The main property of such matrix is its spectral decomposition in terms of the Fourier matrix. 
Let $F_{1}$  be the discrete Fourier matrix defined as $[F_{1}]_{i,j}=\frac{1}{\sqrt{n}}\e^{- \imm \frac{2 \pi ij}{n}}$, for $i,j=0,\dots,n-1$.
The two-dimensional Fourier matrix is defined by tensor product as $F_{2} = F_{1} \otimes F_{1}$ and matrix-vector products with $F_2$ can be computed in $O(n^2\log(n))$ by FFT.
The eigenvalues $\lambda_j$, with $j=1,\dots,n^2$, of $\cir_n(f)$, can be computed by applying $F_{2}$ to the first column of $\cir_n(f)$, which is obtained stacking a proper permutation of the PSF; see \cite{book} for details.
In this way, the matrix $A$ can be factorized as 
 \[
 A=\cir_n(f)=F_{2}^H\Lambda F_{2}, 
 \]
where  $\Lambda$ is the diagonal matrix of the eigenvalues $\lambda_j$, with $j=1,\dots,n^2$.
Note that $\cir_n(f)$ is the Strang preconditioner of $\toep_n(f)$ and the $n^2$ eigenvalues of $\cir_n(f)$ can also be written as 
\begin{equation}\label{eq:cij}
\lambda_{i+jn+1}=f\left(\frac{2\pi i}{n},\frac{2\pi j}{n}\right), \qquad i,j=0,\dots,n-1,
\end{equation}
where $f$ is the generating function $\eqref{genfunct}$; see \cite{Davis}. 

The solution of the linear system \eqref{original} by Tikhonov regularization gives rise to the minimization problem
\[
\min_{\bx\in\R^{n^2}}\|A\bx-\bb\|^2+\alpha\|\bx\|^2,
\]
where $\alpha>0$ is a regularization parameter and $\| \cdot \|$ denotes the Euclidean norm. This minimization problem has the unique solution
\[
\bx_\alpha=(A^TA+\alpha I)^{-1}A^T\bb.
\]
In the case of periodic boundary conditions, since  $A=\cir_n(f)$, the Tikhonov solution $\bx_\alpha$ can be computed by applying three FFTs as
\[
%\bx_\alpha=F_{2}^H\Lambda_{\rm Tik} F_{2} \bb,
\bx_\alpha=\cir_n(p_\alpha)\bb,
\]
%where the diagonal matrix $\Lambda_{\rm Tik}$ has the diagonal entries
where 
\begin{equation}\label{eq:ftik}
	p_\alpha(\vartheta_1,\vartheta_2)=\frac{\overline{f(\vartheta_1,\vartheta_2)}}{|f(\vartheta_1,\vartheta_2)|^2+\alpha}
\end{equation}
and thus the eigenvalues of $\cir_n(p_\alpha)$ are
\[
	\frac{\overline{\lambda_j}}{|\lambda_j|^2+\alpha}, \qquad j=1,\dots,n^2.
\]

%------------------------------------------------------------------------------------------------------------------------------
\section{Regularizing preconditioner for symmetrized BTTB matrices}\label{sec:reg_prec}
Starting from the seminal paper \cite{HNPprec}, regularizing preconditioners have been largely investigated to speed up the convergence of iterative regularization methods without sploiling their reconstructions \cite{claudio,pietro,david}. The idea behind these approaches is to cluster the eigenvalues in the signal subspace without modifying the noise subspace. Moreover, the linear system with the preconditioned matrix has to be solved with a computational cost comparable to the matrix-vector product with the matrix $A$. Therefore, a common choice for the preconditioning of $A$, independently of the imposed boundary conditions, is to use the circulant Tikhonov operator $\cir_n(p_\alpha)$, where $p_\alpha$ is defined in \eqref{eq:ftik}.

We introduce the eigenvalues distribution that will be useful to study the clustering of the eigenvalues of the preconditioned matrices. 
By the Szeg\"o-Tilli theorem \cite{tilli}, a distribution
relation holds for the eigenvalues of the matrix sequence $\{\toep_n(g)\}_n$ {for a real-valued $g~\in~L^1([-\pi,\pi]^2)$}:
\begin{equation}\label{eq:spectral_distribution}
\lim_{n\to\infty}\frac1{n^2}\sum_{j=1}^{n^2}F[\lambda_j(\toep_n(g))]
=\frac1{(2\pi)^2}\int_{[-\pi,\pi]^2}F(g(\theta_1,\theta_2))\,{\rm d}\theta_1\,{\rm d}\theta_2,
\qquad\forall F\in C_c(\mathbb{C},\mathbb{C}),
\end{equation}
where $C_c(\mathbb{C},\mathbb{C})$ denotes the space of continous functions $F: \mathbb{C} \rightarrow \mathbb{C}$.
The function $g$ is called the \emph{symbol} of the Toeplitz family
and we write 
\[
\{\toep_n(g)\}_{n}\sim_\lambda g.
\]
The informal meaning behind the above distribution result is the following.	
If $g$ is continuous and $n$ is large enough, then the spectrum of $\toep_n(g)$ `behaves' like a uniform sampling of $g$ over $[-\pi,\pi]^2.$ It follows that $\cir_n(g)$ is a common preconditioner for $\toep_n(g)$ due to its eigenvalue distribution, cfr. equation \eqref{eq:cij}.

Imposing accurate boundary conditions, as reflective or antireflective, the coefficient matrix is a BTTB matrix
type up to a `small' correction. It turns out that, independently of the imposed boundary conditions, the matrices associated to the PSF have the same symbol $f$ like their Toeplitz part.
	
Since the PSF performs an average of neighboring pixels, we have 
\[ \sum_{i,j=-n+1}^{n-1} h_{i,j} = 1, \qquad h_{i,j} \geq 0.\]
Therefore, the function $|f|$ has maximum at the origin and then decays, not necessarily uniformly, reaching the minimum in $[\pi,\pi]$.
For instance, for the PSF used in Example~\ref{ssec:satellite}, plotted in Figure \ref{fig:symbol_plot} (a), Figure \ref{fig:symbol_plot} (b) depicts the behavior of $|f|$ as described above.
Therefore, the generating function $f$ satisfies the assumption of the following theorem.
\begin{figure}
	\begin{center}
	\begin{subfigure}[c]{0.3\textwidth}
		\includegraphics[width=\textwidth]{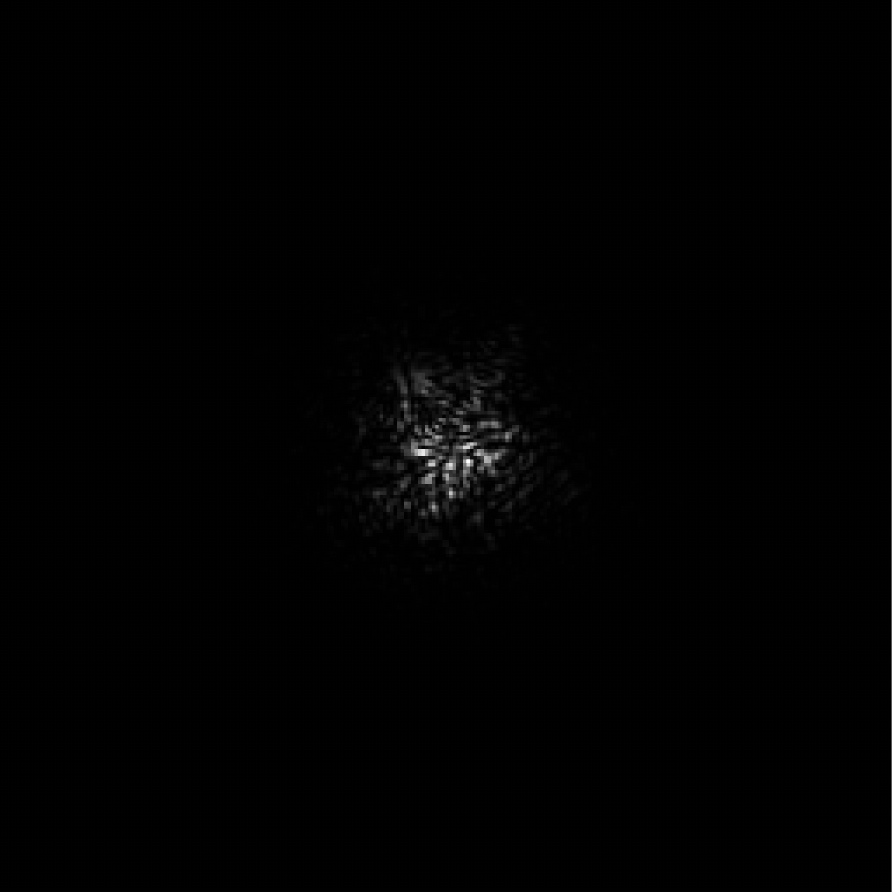}
		\caption{PSF}
	\end{subfigure}
	\begin{subfigure}[c]{\textwidth}
		\includegraphics[width=\textwidth]{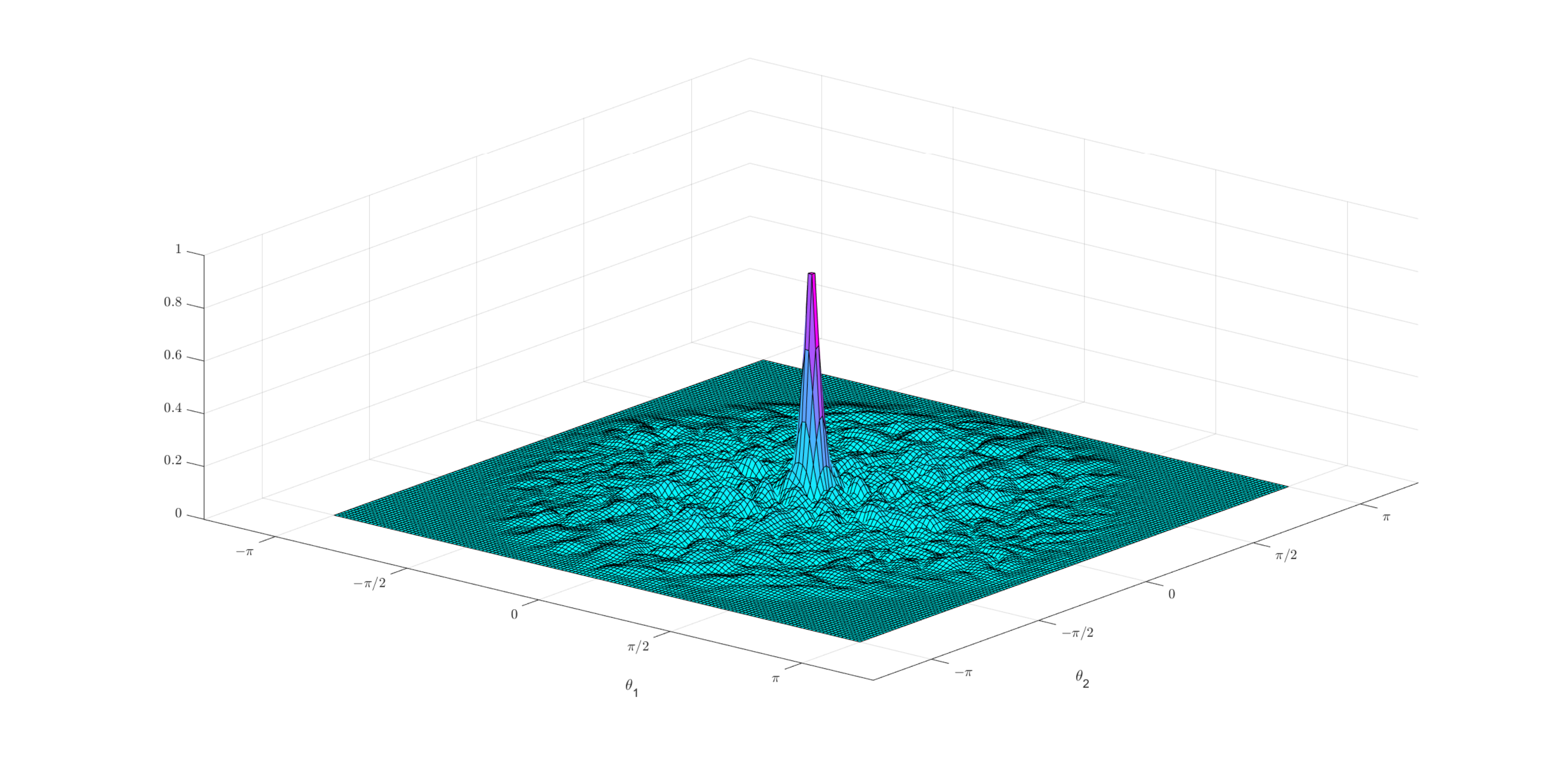}
		\caption{Symbol $|f|$}
	\end{subfigure}
	\caption{PSF of speckle blur and associated symbol.}
	\label{fig:symbol_plot}
	\end{center}
\end{figure}
		
\begin{theorem}\label{thm:preconditioner_cluster}
		Let $\varepsilon$ and $\tau$ be positive values such that $\varepsilon\in(0,1)$ and $\tau\in (0,\pi)$. Let $f\in L^1([-\pi,\pi]^2)$  be a bivariate function with real Fourier coefficients,  periodically extended to the whole real plane, and such that
		\begin{equation}\label{eq:condition_on_f}
			\begin{cases}
				|f(\vartheta_1,\vartheta_2)| > {\varepsilon}, &  \mbox{if }\left|\vartheta_1^2+\vartheta_2^2\right|<\tau,\\
				|f(\vartheta_1,\vartheta_2)| \le \varepsilon, 		& \mbox{otherwise}
			\end{cases}.
		\end{equation}
		and define 
		\begin{equation}
			g_\tau(\vartheta_1,\vartheta_2) = 
			\begin{cases}
				|{f(\vartheta_1,\vartheta_2)}|, &  \mbox{if }\left|\vartheta_1^2+\vartheta_2^2\right|<\tau,\\
				1, 		& \mbox{otherwise}
			\end{cases}.
		\end{equation}
%		Let ${C}_\mathbf{n}$ be a bi-level circulant matrix such that $\left\{{C}_\mathbf{n}\right\}_\mathbf{n} \sim_{\mbox{GLT}}  g$. Then
%		\begin{equation}
%			\left\{{C}_\mathbf{n}^{-1}YT_\mathbf{n}(f)\right\}_\mathbf{n} \sim_\lambda  \psi
%		\end{equation}
Then
		\begin{equation}
			\left\{\cir_n(g_\tau)^{-1}Y\toep_n(f)\right\}_n \sim_\lambda  \psi
		\end{equation}
where
		\begin{equation}\label{eqn:prec_distribution}
			\psi(\vartheta_1,\vartheta_2) = 
			\begin{cases}
				\frac{1}{|f|}\begin{bmatrix} 0 & f(\vartheta_1,\vartheta_2) \\ \\ \overline{f(\vartheta_1,\vartheta_2)} & 0 \end{bmatrix}, &  \mbox{if }\left|\vartheta_1^2+\vartheta_2^2\right|<\tau,\\ \\
				\begin{bmatrix} 0 & f(\vartheta_1,\vartheta_2) \\ \\ \overline{f(\vartheta_1,\vartheta_2)} & 0 \end{bmatrix}, & \mbox{otherwise}
			\end{cases}.
		\end{equation}
	\end{theorem}
	\begin{proof}
		See Appendix \ref{appendixGLT}.
	\end{proof}
%	\begin{proof}
%		Defining $\psi=\frac{1}{h}\begin{bmatrix} 0 & f \\ \overline{f} & 0 \end{bmatrix}$, on $[-\pi,\pi]^2 \backslash [-\tau,\tau]^2$ we have $\psi=\begin{bmatrix} 0 & f \\ \overline{f} & 0 \end{bmatrix}$ whereas on on $[-\tau,\tau]^2$ we have
%		\[
%			\psi = \frac{1}{|f|}\begin{bmatrix} 0 & f \\ \overline{f} & 0 \end{bmatrix}=
%		\]
%		which concludes the proof.
%	\end{proof}
\begin{remark}
		On the set where $|f|<\varepsilon$, we have $\psi=\begin{bmatrix} 0 & f \\ \overline{f} & 0 \end{bmatrix}$, which has eigenvalue functions $\pm |f|$ with image contained in $(-\varepsilon,\varepsilon)$. If $|f|> {\varepsilon}$, we have $\psi = \frac{1}{|f|}\begin{bmatrix} 0 & f \\ \overline{f} & 0 \end{bmatrix}$ with eigenvalue functions $\pm 1$. Hence, according to Remarks 2.12--2-13 in \cite{mbg}, spectral distribution~(\ref{eqn:prec_distribution}) suggests that there are at most $o(n^2)$ eigenvalues of $\cir_n(g_\tau)^{-1}Y\toep_n(f)$ outside three clusters at $\lambda=1$, $\lambda=-1$, and $\lambda\in[-\varepsilon,\varepsilon]$, whose cardinality depends on $\varepsilon$ and $\tau$, provided that $n$ is large enough. Note that $\varepsilon$ and $\tau$ are not independent: for each $\varepsilon$ you find a $\tau$ such that \eqref{eq:condition_on_f} holds. If $\varepsilon$ tends to 0, $\tau$ tends to $\pi$, that is, the set where $|f|<\varepsilon$ has the empty set as a limit, making the eigenvalues of the preconditioned matrix clustered at $\pm 1$. On the other hand, if $\varepsilon$ is chosen greater than the maximum of $|f|$, the preconditioner has no effect at all. So, it is important to choose $\varepsilon$, and consequently $\tau$, accurately, so that the preconditioner has a signifincant clustering effect without amplifying the noise.
\end{remark}

\begin{remark}	
	As we have already stated, a common choice for the preconditioning of $A$ in the image deblurring context is $\cir_n(p_\alpha)$, where $p_\alpha$ is defined in \eqref{eq:ftik}. In \cite{doi:10.1137/140974213,FFHMS,MR4304085} it is shown that under proper assumptions, if $\cir_n(p_\alpha)$ preconditions $A$, the absolute value of $\cir_n(p_\alpha)$ preconditions $YA$. For this reason, we consider the circulant matrix $\cir_n(|p_\alpha|)$ and for brevity we define $\tilde{g}_\alpha=|p_\alpha|$:
	\[
		\tilde{g}_\alpha(\vartheta_1,\vartheta_2)=\frac{|f(\vartheta_1,\vartheta_2)|}{|f(\vartheta_1,\vartheta_2)|^2+\alpha}.
	\]
	In general, the function $g_\tau$ is discontinuous and this guarantees a sharp subdivision between eigenvalue clusters. For $\tilde{g}_\alpha$ this is not true since it is a smooth low-pass filter. This implies that $\cir_n(\tilde{g}_\alpha)$ is less sensitive than $\cir_n(g_\tau)$ to the choice of the threshold parameter $\alpha$ and $\tau$, respectively, which is related to $\varepsilon$.
 
 To prove that $\cir_n(\tilde{g}_\alpha)$ is a regularizing preconditioner, suppose for simplicity that we take $\alpha = \varepsilon$ and study $\tilde{g}_\varepsilon$.
	With the considerations that we make in Appendix \ref{appendixGLT}, we can state that
	\[
		\left\{\cir_n(\tilde{g}_\alpha) Y\toep_n(f)\right\}_n \sim_\lambda  \tilde{g}_\alpha \begin{bmatrix} 0 & f \\ \overline{f} & 0 \end{bmatrix},
	\]
	which has eigenvalue functions
	\[
		\frac{|f|}{|f|^2+\varepsilon}(\pm |f|) = \pm \frac{|f|^2}{|f|^2+\varepsilon}.
	\]
	Note that $\frac{|f|^2}{|f|^2+\varepsilon}<1$, and when $|f|$ is much greater than $\varepsilon$ the eigenvalues are close to 1, which means a speed up of the convergence in the signal subspace.
	
	On the other hand, if $|f|\le \varepsilon$, we have
	\[
		\frac{|f|^2}{|f|^2+\varepsilon}\le \frac{\varepsilon^2}{|f|^2+\varepsilon}\le \frac{\varepsilon^2}{\varepsilon} = \varepsilon,
	\]
	which means that the small eigenvalues are not amplified by the preconditioner.
	
	In the numerical examples we will use $\tilde{g}_\alpha$ instead of ${g}_\tau$ since it showed an overall greater robustness and better performance.
\end{remark}
	
%------------------------------------------------------------------------------------------------------------------------------	
\section{Iterative regularization methods}\label{sec:itReg}

The purpose of the present section is to briefly introduce the iterative methods that we will use as regularization methods in the numerical experiments section. As mentioned in Section \ref{sec:intro}, iterative methods used to reconstruct blurred and noisy images exploit \emph{semiconvergence}, i.e., they firstly reduce the error and then diverge from the exact solution. When possible, we will use the discrepancy principle as stopping criterion for the iterations, that is, we stop the iterations when the norm of the residual vector is less than a tolerance times the noise level, i.e., the norm of the unknown noise vector affecting the right-hand-side vector $\bb$ in \eqref{original}. 

As we anticipated, in the case where $A$ is a BTTB matrix, we want to study the behaviour of MINRES applied to the linear system \eqref{eq:Ysys}, which has a symmetric coefficient matrix. The regularization properties of the MINRES method were proven in \cite{KilmerMINRES, MR2312499}. When using MINRES, a preconditioning strategy needs to be applied symmetrically (i.e., both on the left and on the right), that is, the preconditioned symmetrized linear system becomes
\[
	P^{1/2}YAP^{1/2}\bz=P^{1/2}Y\bb, \; \bx=P^{1/2}\bz.
\]
% In this case we have both left and right preconditioning. The residual of the right-preconditioned system is the same of that of the non-preconditioned one. Hence, right preconditioning allows us to apply the discrepancy principle without a significant additional computational cost, since an efficient implementation of the MINRES computes the iteration residual without requiring an additional matrix-vector product. Conversely, 
While when applying right preconditioning only the residual of the preconditioned system is the same of that of the non-preconditioned one and the discrepancy principle can be applied without a significant additional computational cost, left preconditioning alters the residual. Therefore, in the symmetric preconditioning case, we will comment on the best reconstruction achieved by the considered methods and their stability. Other stopping criteria can be chosen, but the study of their behaviour is beyond the purpose of the present paper.

While theoretical results guarantee a regularizing behaviour of MINRES, the success of GMRES as regularization method is problem dependent. Although some theory has been developed \cite{calvetti}, it often happens that the GMRES is not effective when the coefficient matrix $A$ is not close to normal \cite{MR2312499}. In the case where we consider reflective boundary conditions, if the PSF is not quadrantally symmetric, the matrix $A$ is neither symmetric nor persymmetric. However, $YA$ is close to being symmetric even if $A$ is highly non-symmetric. In this case, we expect that GMRES applied to the system \eqref{eq:Ysys} over-performs GMRES applied to the original system.

The LSQR method requires $A^T$, and replacing it with $A'$ is easy to implement but is not theoretically sound, as explained in the introduction. 
We stress that one iteration of LSQR costs a matrix-vector product with matrix $A^T$ more than one iteration of MINRES and GMRES, so this needs to be taken into account when analysing the convergence speed in terms of iteration number.

% \todo[inline]{Silvia: scegli tu se aggiungere dettagli sul FGMRES qui oppure nella sezione numerica}
In order to speed up the convergence of the iterative methods listed above, we apply the preconditioning strategy analysed in Section \ref{sec:reg_prec}. More precisely, we use $\cir_n(p_\alpha)$ as a (right) preconditioner for LSQR type methods, which we apply to systems with $A$, and $\cir_n(|p_\alpha|)$ for GMRES type methods, which we apply to systems with $YA$, on the right. In particular, when we use the circulant strategy in combination with FLSQR and FGMRES, the preconditioner becomes the iteration-dependent matrix $\cir_n(p_{\alpha_k})$, where we choose the parameter as the geometric sequence $\alpha_k=\alpha_0 q^k$, where $k$ is the iteration counter, while $\alpha_0 = 0.1$ and $q=0.8$ are chosen following the recommendations in \cite{martin}. In the rest of the paper, when we denote a preconditioner by $P$, we mean a circulant preconditioner, whose exact formulation will be clear from the context. 

For enforcing sparsity in the computed solution, we apply the iteration-dependent preconditioner studied in \cite{silvia14, MR4331956}, which here we simply denote by $W$. More specifically, at the $k$th iteration of the considered methods we have 
\[
W = \text{diag}(|\bx_{k-1}|^{1/2})\,,
\]
where $\bx_{k-1}$ is the solution computed at the previous (i.e., the $(k-1)$th) iteration, and where absolute value and exponentiation are applied component-wise. 
Such preconditioner stems from the iteratively reweighted least square method applied to the problem
\[
\min_{\bx}\|A\bx-\bb\|_2^2 + \lambda \|\bx\|_1\,,
\]
whereby, assuming that the matrix $W$ is invertible\footnote{The matrix $W^{-1}$ is required for the formal change of variable, but in practice the application of the preconditioner requires only a matrix-vector product with the matrix $W$. Therefore, we do not implement any shifting strategy to guarantee the invertibility of $W$ when $\bx_{k-1}$ has null entries.}, the regularization term $\lambda \|\bx\|_1$ is first approximated by the squared 2-norm term $\lambda \|W^{-1}\bx\|_2^2$ and then undergoes a transformation to standard form (change of variable), leading to the problem
\[
\min_{\by}\|AW\by-\bb\|_2^2 + \lambda \|\by\|_2^2\,,\quad \by=W^{-1}\bx\,.
\]
Following common practice, we apply either FGMRES or FLSQR to the above problem, after dropping the regularization term. Such methods allow updates of $W$ as the iterations proceed. The iteration-dependent matrix $W$ modifies the approximation subspace for the solution in such a way that sparsity is naturally enforced within its basis vectors. We refer to \cite{newsurvey, MR4331956} for more details about these approaches; a numerical illustration is also provided in Section \ref{ssec:edges}. 

\section{Numerical Examples}\label{sec:numExp}
In this section we present four examples. The Satellite test problem in Subsection \ref{ssec:satellite} and the Phantom test problem in \ref{ssec:phantom} are aimed at discussing the efficiency of the symmetrization strategy in the case of zero boundary conditions, namely, when the matrix $A$ is BTTB. The Cameraman test problem in Subsection \ref{ssec:cameraman} is an example with reflective boundary conditions, which shows the performance of the `symmetrization' strategy on matrices that are close to being symmetric. Subsection \ref{ssec:edges} is devoted to the numerical study of the reconstruction of a sparse image, showing the effect of the `symmetrization' strategy and preconditioning within flexible Krylov subspace methods. 

In all examples, the blurred and noisy image is given by
	% \[\bb=A\bx+\sigma\frac{\mathbf{e}}{\|\mathbf{e}\|}\|A\bx\|,\] 
 \[\bb=A\bx+\frac{\bxi}{\|\bxi\|}\sigma\|A\bx\|,\] 
	where $\bx$ is the exact image, $\bxi$ %$\mathbf{e}$ 
 is a random Gaussian white noise vector and $\sigma$ is the noise level.
When available, we use the blurring functions and the implementation of the iterative methods included in \IRTools\ \cite{ir}.	
	
\subsection{Test Problem: Satellite with zero boundary conditions}\label{ssec:satellite}
	The satellite image is blurred with a medium level speckle blur; a noise of level $\sigma=0.05$ is added. Figure~\ref{fig:sat_images} shows the exact image, the PSF, and the blurred image, all of size $256\times 256$ pixels.
		\begin{figure}
		\includegraphics[width=0.3\textwidth]{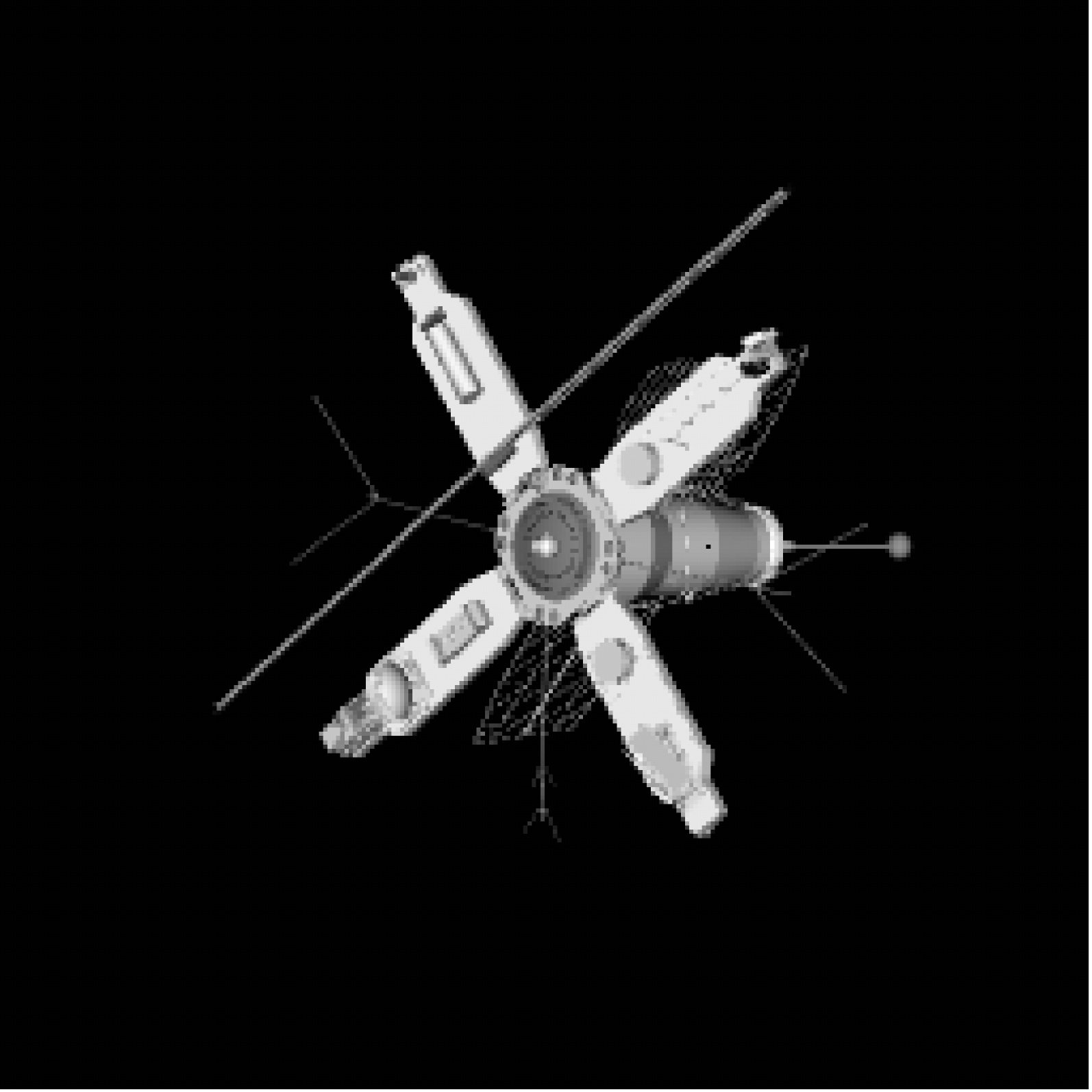}
		\includegraphics[width=0.3\textwidth]{Figures/sat_image_psf.pdf}
		\includegraphics[width=0.3\textwidth]{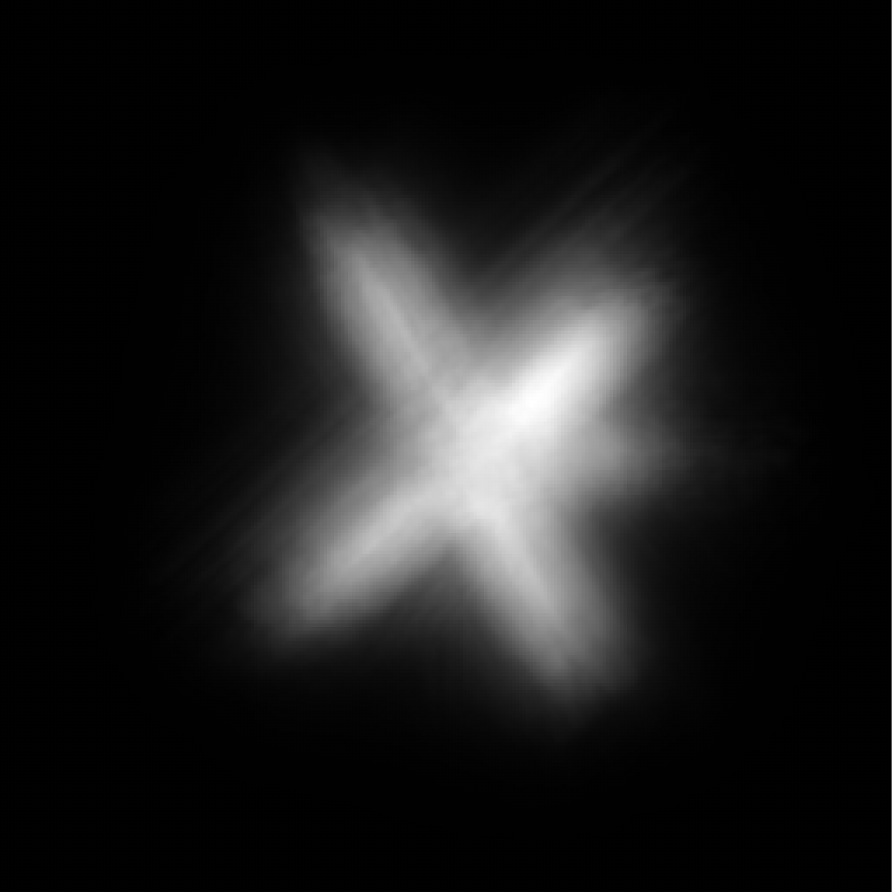}
		\caption{Satellite test problem. From left to right: exact image, PSF, and blurred image.}
		\label{fig:sat_images}
	\end{figure}
	In Figure~\ref{fig:errors_sat_prec}, we report the error behaviour of (preconditioned) MINRES applied to the symmetrized system \eqref{eq:Ysys} and of (preconditioned) LSQR and GMRES applied to the non-symmetrized linear system. The regularizing parameter $\alpha$ needed in the construction of the circulant preconditioners $\cir_n(p_\alpha)$ and $\cir_n(|p_\alpha|)$ with $p_\alpha$ defined in \eqref{eq:ftik} was manually tuned, i.e., we chose the $\alpha\in\{10^{-1}, 10^{-2}, 10^{-3}, 10^{-4}, 10^{-5}\}$ which numerically proved to be the best in terms of balancing the convergence speed and the quality of the reconstruction, which is $\alpha=10^{-2}$. MINRES is slower to reach the minimizer than LSQR in terms of iteration number, however one iteration of MINRES is computationally less costly, not requiring the multiplication by $A^T$. Moreover, we highlight that preconditioned MINRES is more stable than preconditioned LSQR: for the latter, the iteration should be stopped between 5 and 12 to achieve a result close to the best possibile, while for preconditioned MINRES the range of possible stopping iteration numbers for achieving an analogous result is between 10 and 22, which is a wider range. GMRES applied to the non-symmetrized linear system does not provide a reconstruction of high quality when compared to the other methods. In Figure \ref{fig:sat_images_reconstructed} we compare the best reconstruction by preconditioned MINRES and LSQR for this difficult deblurring problem.
	\begin{figure}
		\includegraphics[width=0.9\textwidth]{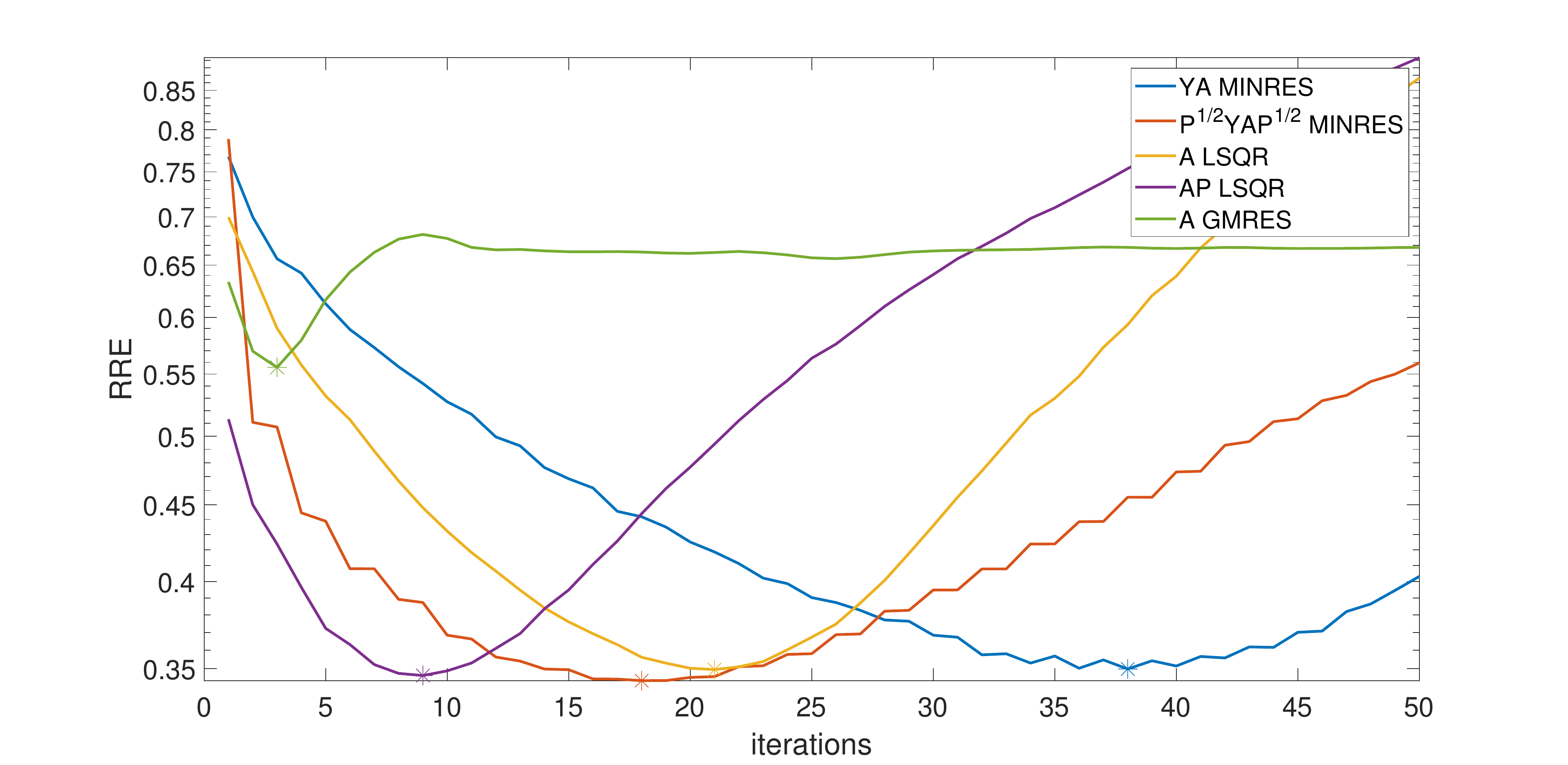}
		\caption{{Satellite} test problem. Comparison between the error behaviours of (preconditioned) MINRES applied to the symmetrized system \eqref{eq:Ysys} and of (preconditioned) LSQR and GMRES applied to the non-symmetrized linear system. The asterisks mark the iterations giving the best reconstruction in terms of RRE.}
		\label{fig:errors_sat_prec}
	\end{figure}
	\begin{figure}
		\includegraphics[width=0.3\textwidth]{Figures/sat_image_true.pdf}
		\includegraphics[width=0.3\textwidth]{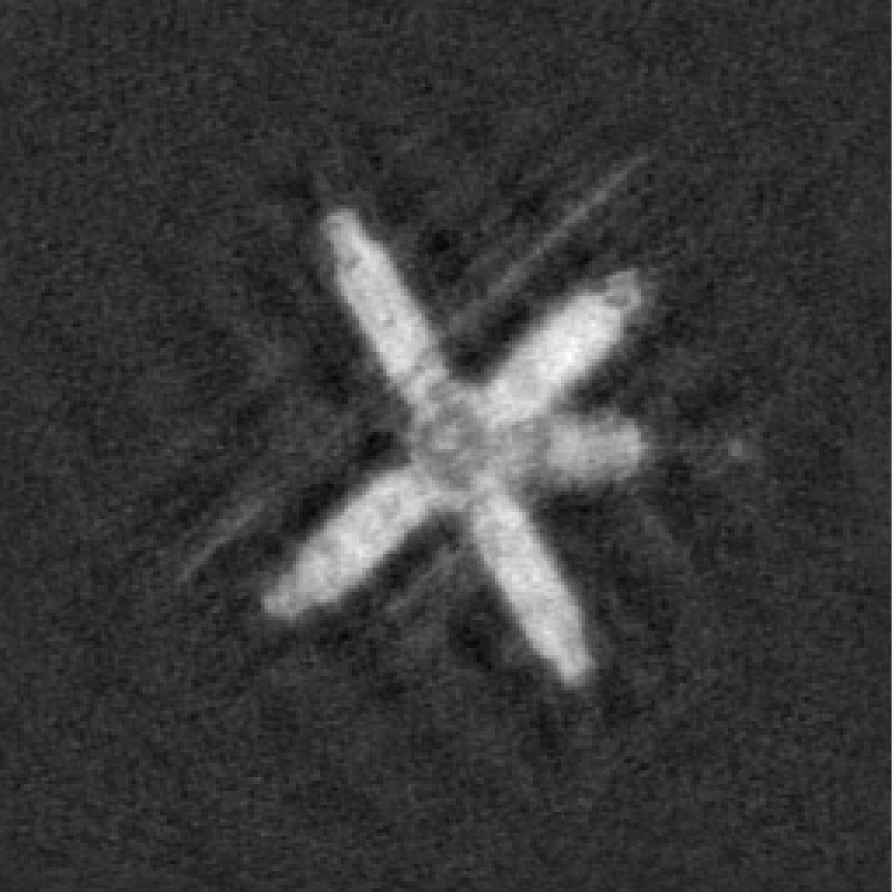}
		\includegraphics[width=0.3\textwidth]{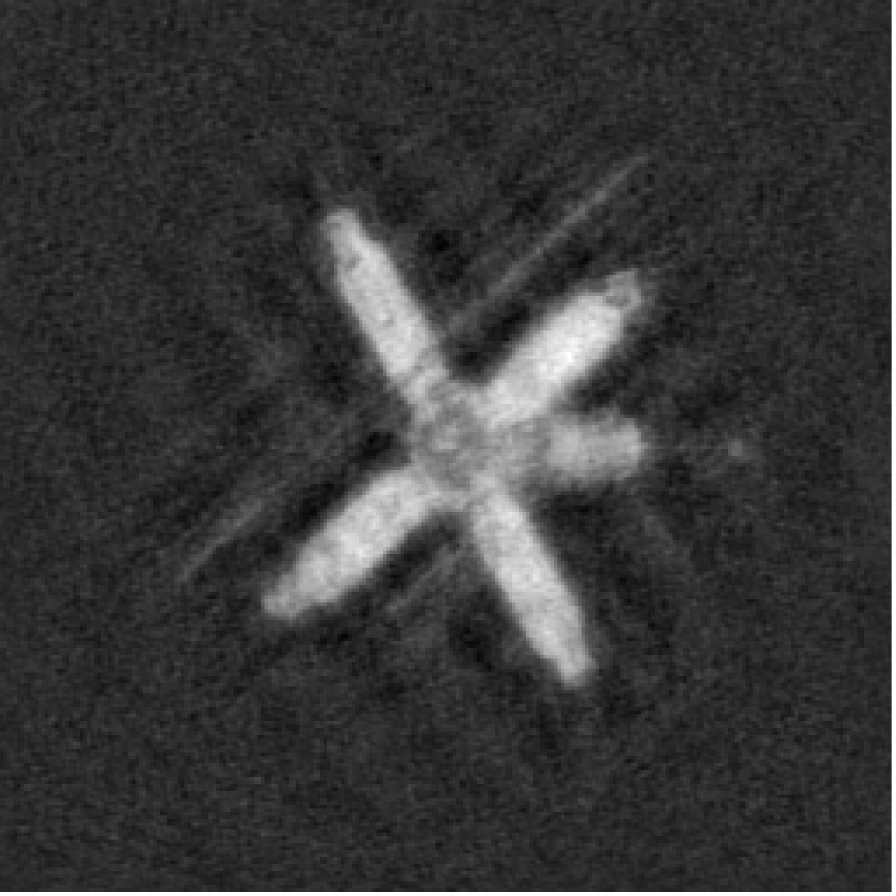}
		\caption{Satellite test problem. Exact image (left), best preconditioned MINRES reconstruction (center), best preconditioned LSQR reconstruction (right).}
		\label{fig:sat_images_reconstructed}
	\end{figure}
	
\subsection{Test Problem: Phantom with zero boundary conditions}\label{ssec:phantom}
	The Phantom test problem analyses the modified Shepp-Logan phantom blurred with a unidirectional motion. The PSF can be seen in Figure \ref{fig:phantom_images}, together with the exact image and the blurred image. A noise of level $\sigma=0.01$ is applied.
	\begin{figure}
		\includegraphics[width=0.3\textwidth]{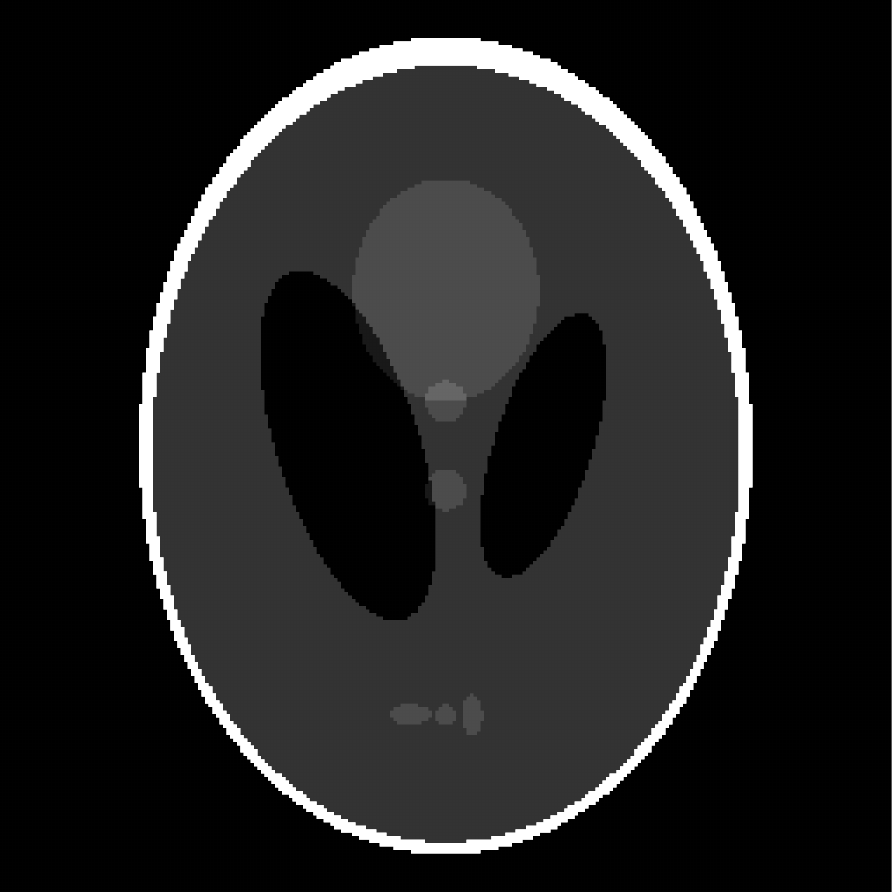}
		\includegraphics[width=0.3\textwidth]{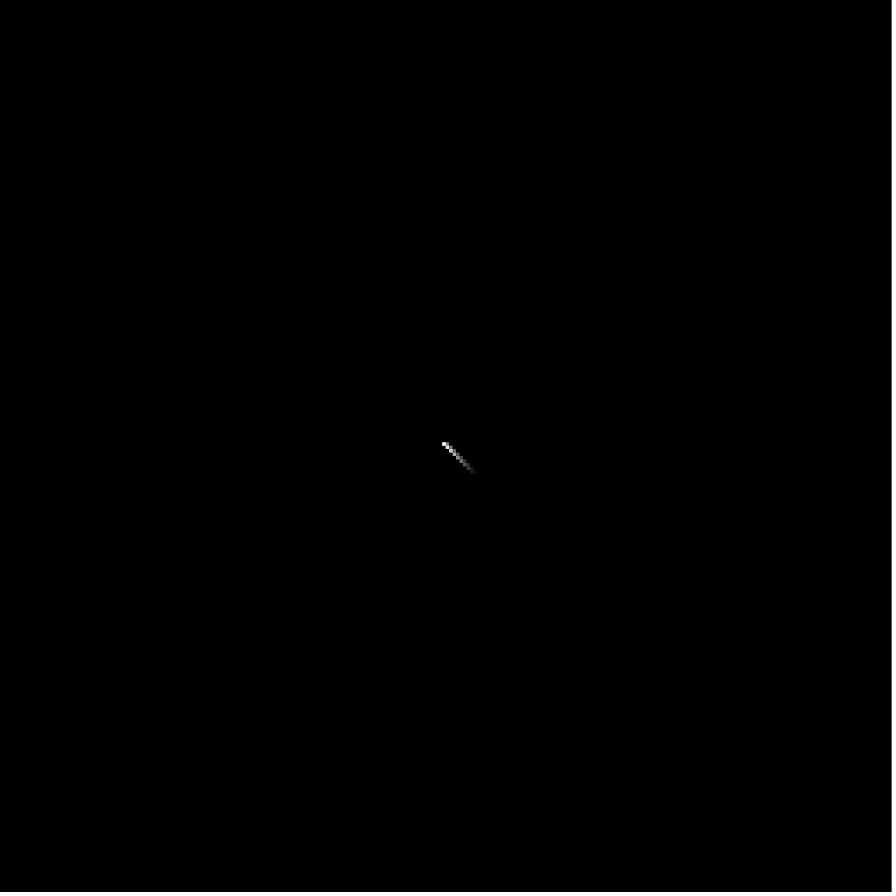}
		\includegraphics[width=0.3\textwidth]{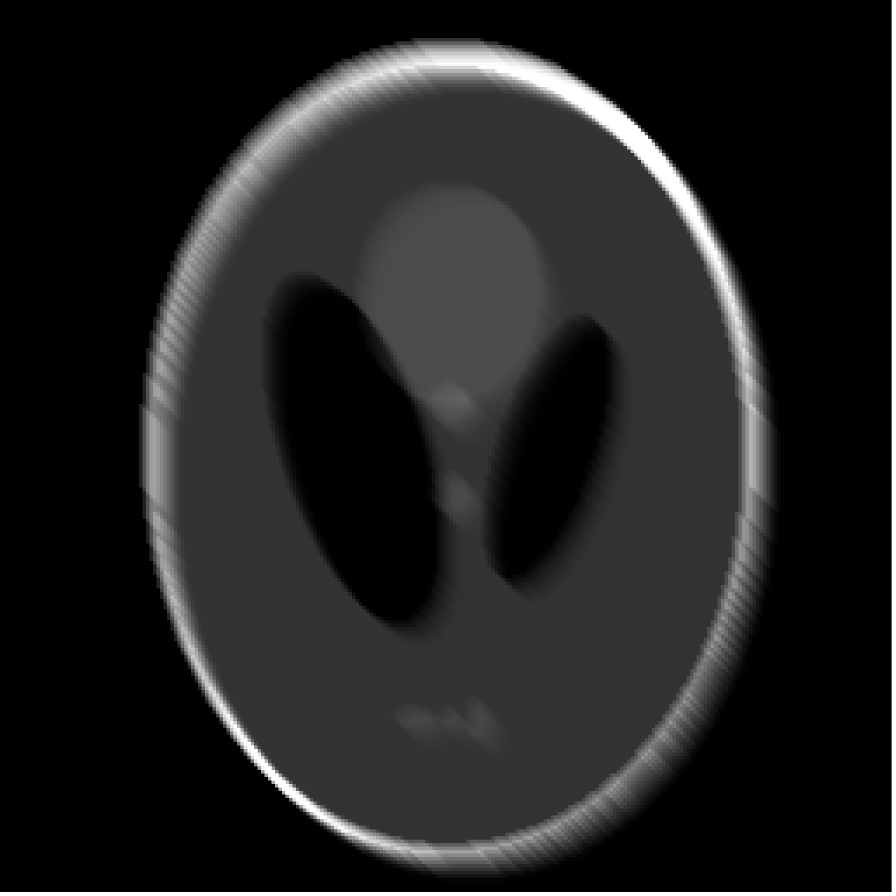}
		\caption{Phantom test problem. From left to right: exact image, PSF, and blurred image.}
		\label{fig:phantom_images}
	\end{figure}
	In Figure \ref{fig:errors_phant_prec}, we report the error behaviour of (preconditioned) MINRES applied to the symmetrized system \eqref{eq:Ysys} and of (preconditioned) LSQR and GMRES applied to the non-symmetrized linear system. Also in this case, the regularizing parameter $\alpha$ needed in the construction of the preconditioner $\cir_n(p_\alpha)$ was manually chosen equal to $10^{-2}$.
	This example confirms that preconditioned MINRES is more stable than preconditioned LSQR, since its semi-convergence is slower. We remark again that this does not translate into a higher computational cost of the overall method, because the cost of a single iteration needs to be taken into account. In the Phantom case, GMRES applied to the non-symmetrized linear system performs better than in the Satellite case, but it is still over-performed by the other methods. In Figure \ref{fig:phant_images_reconstructed} we compare the best reconstruction by preconditioned MINRES and LSQR.
	\begin{figure}
		\includegraphics[width=0.9\textwidth]{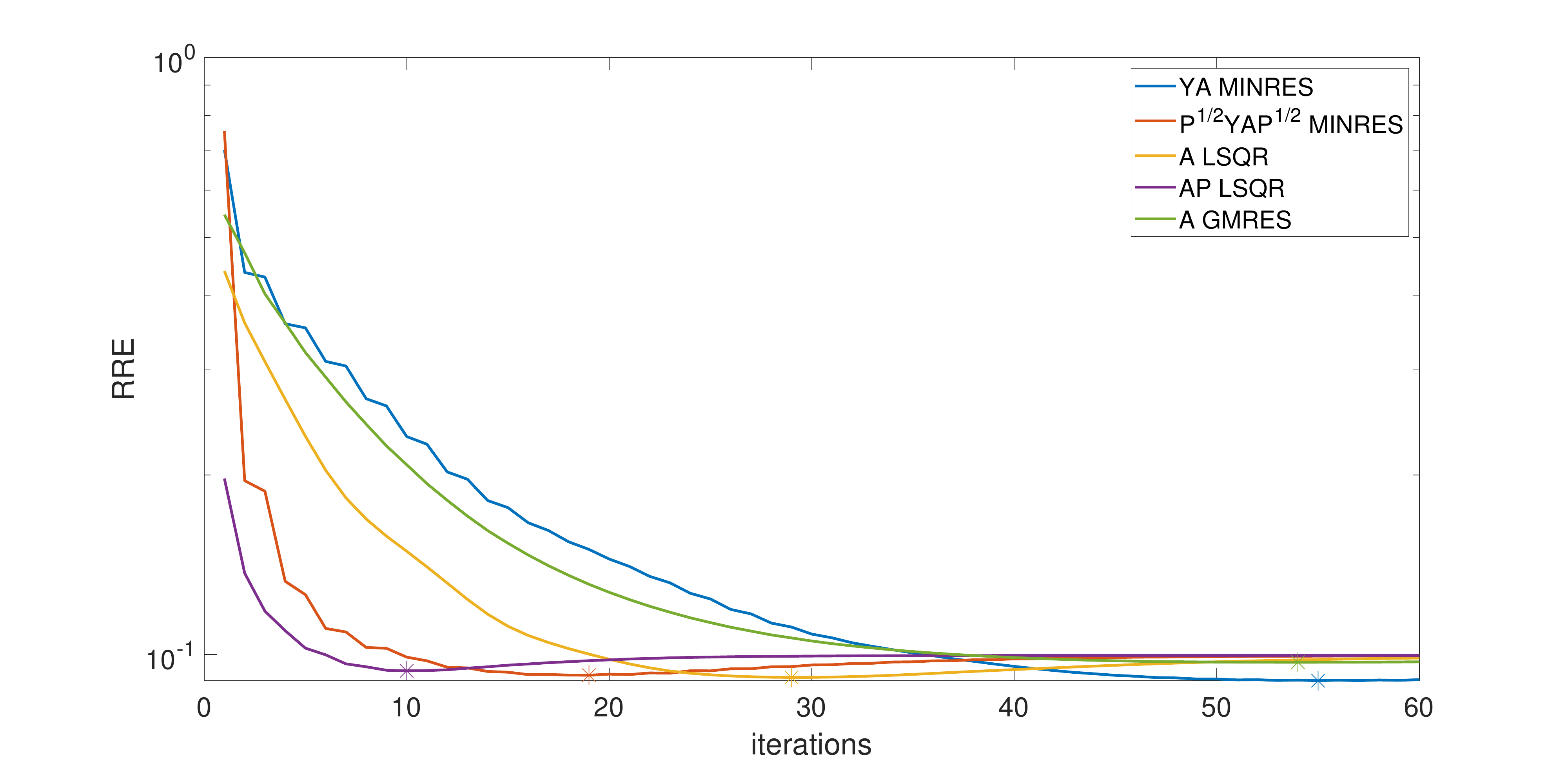}
		\caption{{Phantom} test problem. Comparison between the error behaviours of (preconditioned) MINRES applied to the symmetrized system \eqref{eq:Ysys} and of (preconditioned) LSQR and GMRES applied to the non-symmetrized linear system. The asterisks mark the iterations giving the best reconstruction in terms of RRE.}
		\label{fig:errors_phant_prec}
	\end{figure}
	\begin{figure}
		\includegraphics[width=0.3\textwidth]{Figures/phantom_image_true.pdf}
		\includegraphics[width=0.3\textwidth]{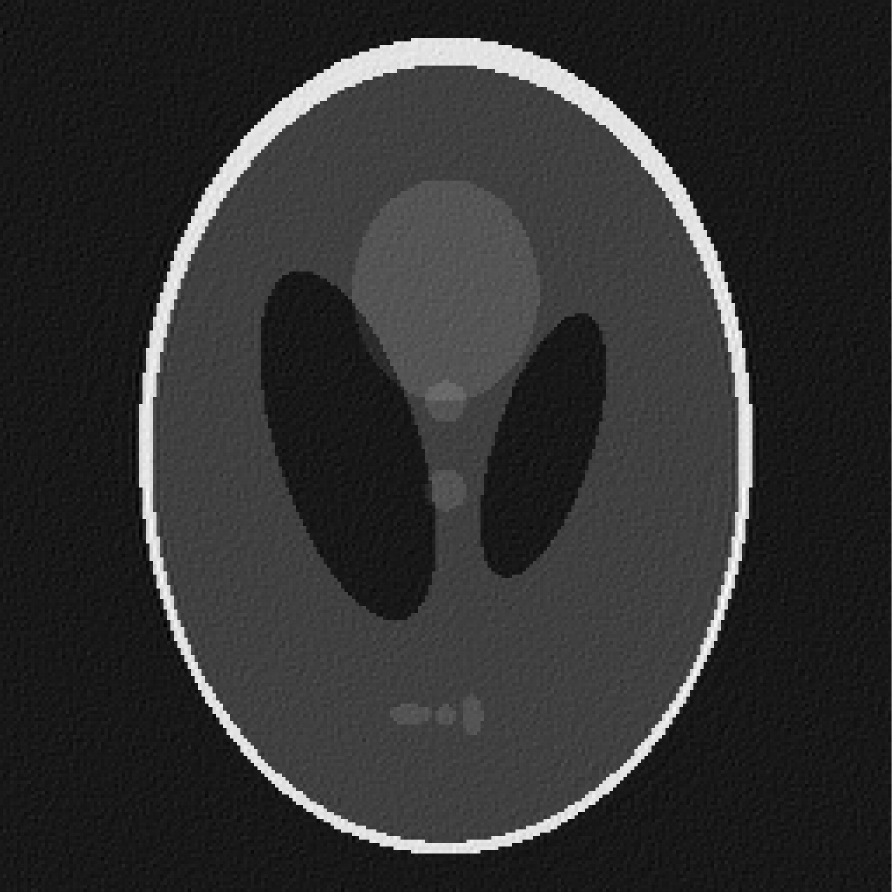}
		\includegraphics[width=0.3\textwidth]{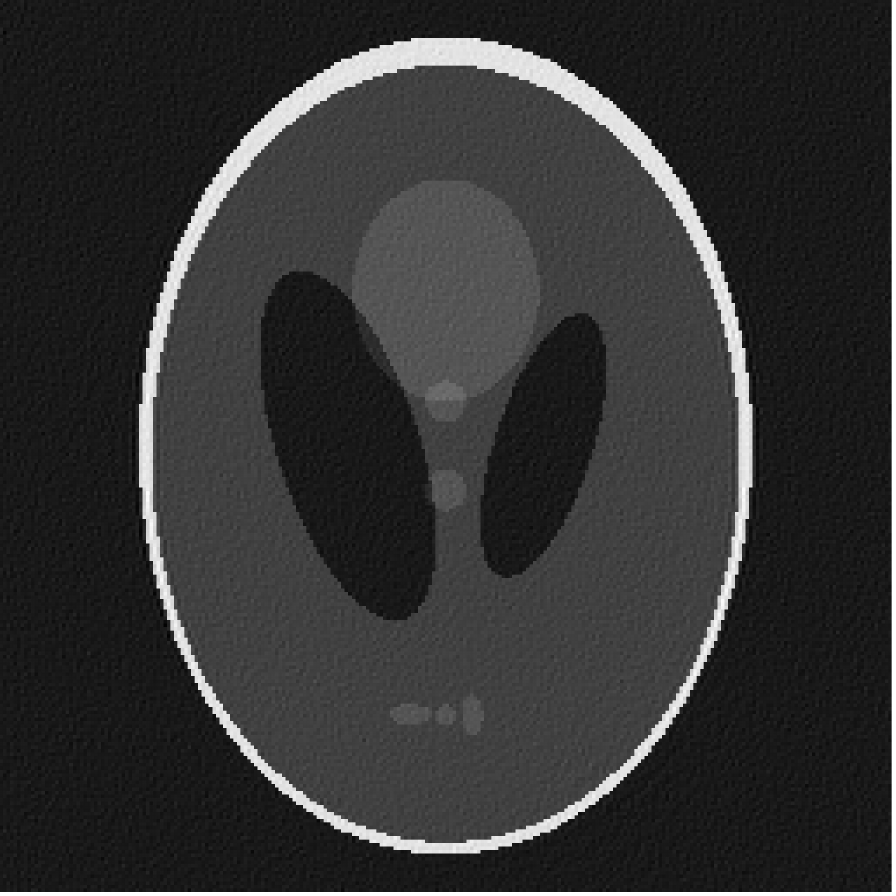}
		\caption{Phantom test problem. Exact image (left),  best preconditioned MINRES reconstruction (center), best preconditioned LSQR reconstruction (right).}
		\label{fig:phant_images_reconstructed}
	\end{figure}
	
%	In Figure \ref{fig:errors_prec_gcv} we analyse the PMR-II preconditioned with $P_{\rm MR-II}(\alpha)$ comparing the error behaviours between the case of $\alpha$ given by GCV to the case of manual best choice.
%	
%		\begin{figure}
%		\includegraphics[width=0.9\textwidth]{Figures/errors_prec_gcv.eps}
%		\caption{Comparison between the MR-II and PMR-II error behaviours when applied to the symmetrized system $YTx=Yb^{\delta}$ for the Phantom test case. For the construction of $P_{\rm MR-II}(\alpha)$ we considered the cases $\alpha$ given by GCV and $\alpha=10^{-2}$.}
%		\label{fig:errors_prec_gcv}
%	\end{figure}	

		\subsection{Test Problem: Cameraman with reflective boundary conditions}\label{ssec:cameraman}
	In this case, we analyse the reconstruction of the cameraman image contaminated by motion blur in two directions and by a noise of level $\sigma=0.01$. The PSF can be seen in Figure \ref{fig:camera_images}, together with the exact image and the blurred image.
		\begin{figure}
		\includegraphics[width=0.3\textwidth]{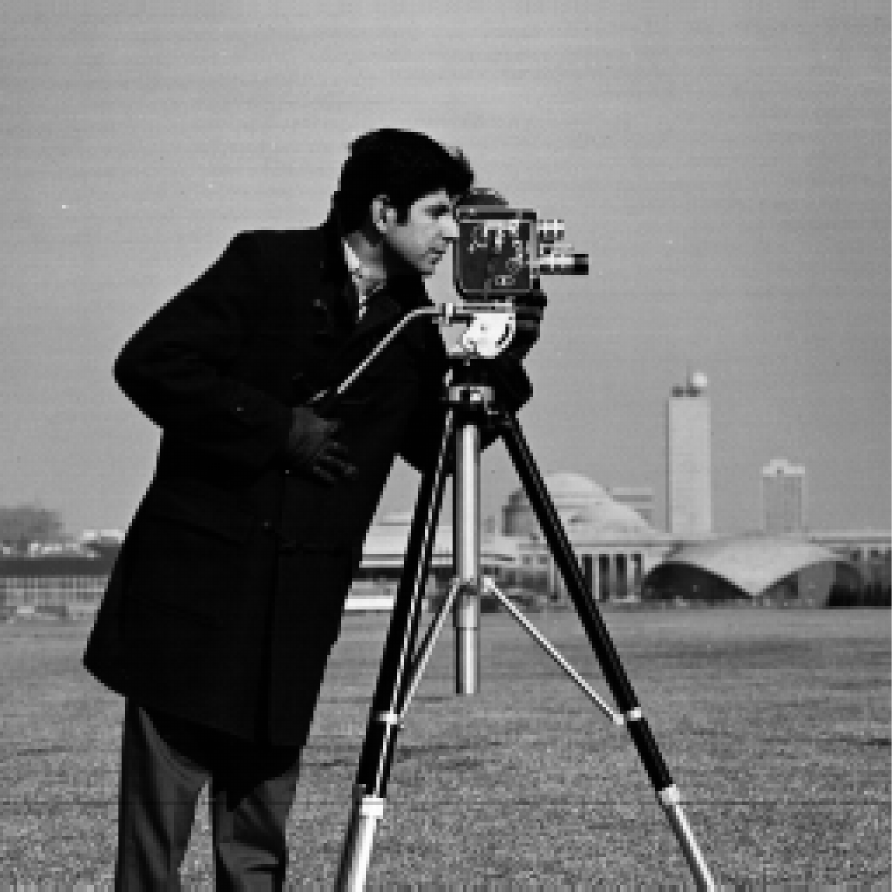}
		\includegraphics[width=0.3\textwidth]{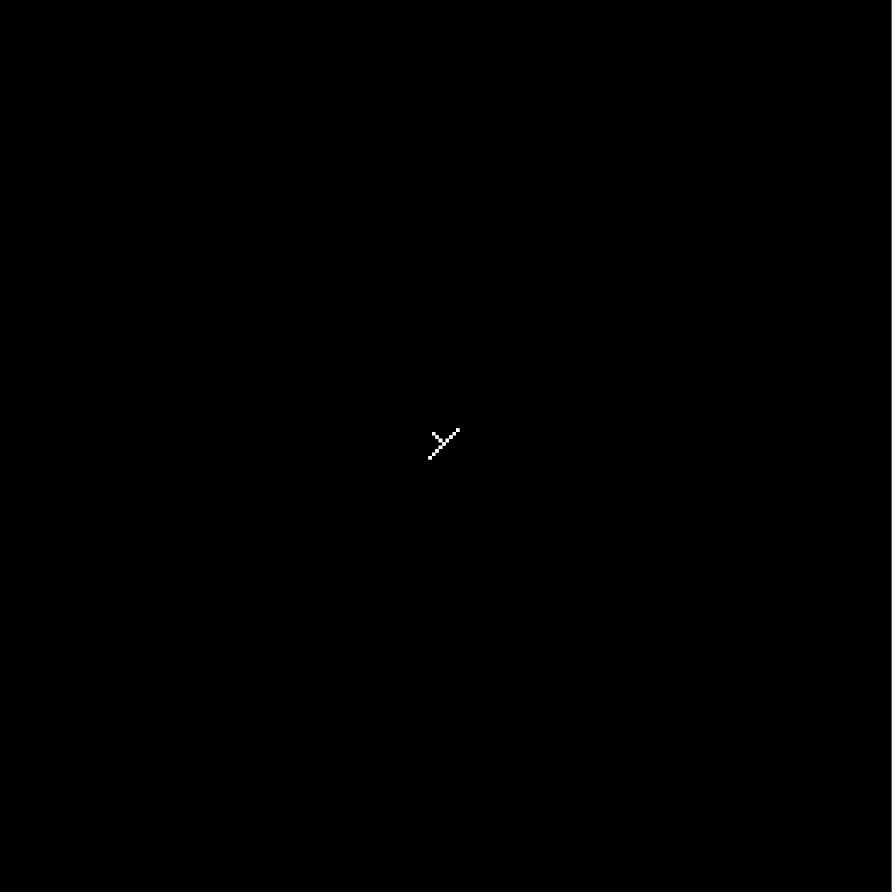}
		\includegraphics[width=0.3\textwidth]{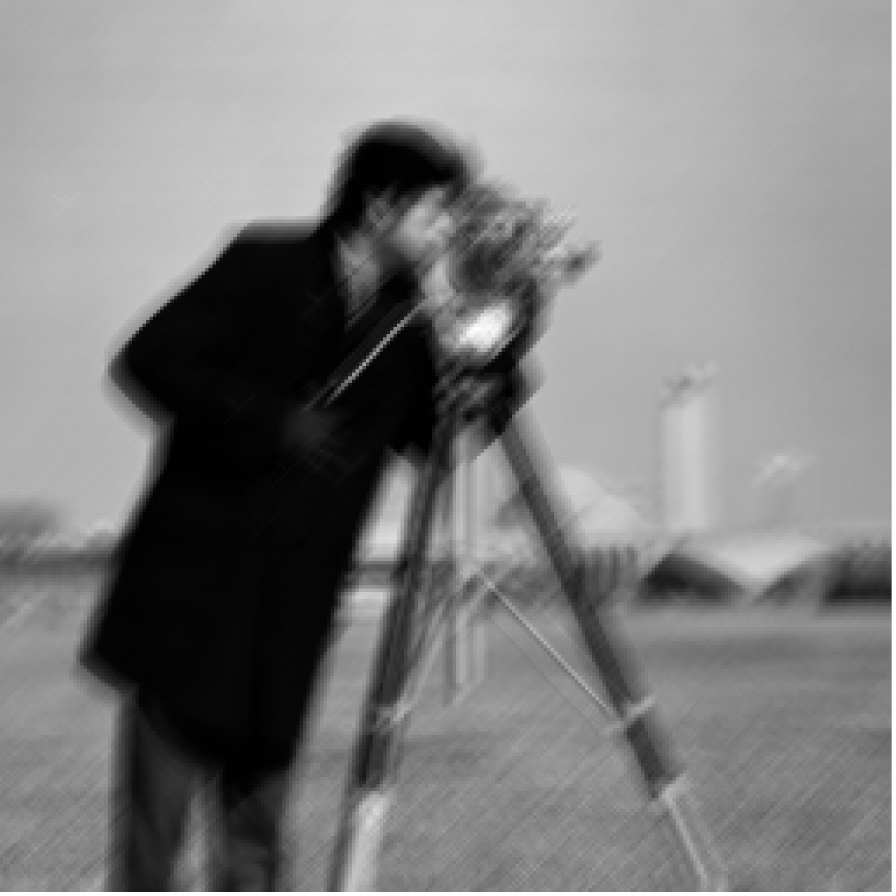}
		\caption{Cameraman test problem. From left to right: exact image, PSF, and blurred image.}
		\label{fig:camera_images}
	\end{figure}
	In Figure \ref{fig:errors_camera_prec}, we report the error behaviour of GMRES and FGMRES applied to the symmetrized system \eqref{eq:Ysys} and of LSQR, FLSQR, and GMRES applied to the non-symmetrized linear system. The iteration-dependent preconditioners for FGMRES and FLSQR are, respectively, $\cir_n(|p_{\alpha_k}|)$ and $\cir_n(p_{\alpha_k})$, with $\alpha_k=0.1\cdot(0.8)^k$. The dots in Figure \ref{fig:errors_camera_prec} show the iteration for which the discrepancy principle is fulfilled. Table \ref{tab:times_psnr_camera} shows the RRE and PNSR values with the corresponding iteration numbers and computational times for the restoration with minimum RRE and for the restoration determined when terminating the iterations with the discrepancy principle. The error behaviour is in accordance with the computational times, since one iteration of LSQR costs approximatively as two iterations of GMRES. We see that for FGMRES the circulant preconditioning strategy accelerates the semi-convergence, while for FLSQR it fails to do so. Of course, the behaviour of the preconditioner strictly depends on the choice of $\alpha_k$, for which more sophisticated strategies can be adopted, but this is beyond the scope of this presentation. In Figure \ref{fig:camera_images_reconstructed} we compare the best reconstruction by FGMRES and FLSQR and we notice that FLSQR produces some artifacts on the boundary, which are not present in the FGMRES reconstruction.
	\begin{figure}
		\includegraphics[width=0.9\textwidth]{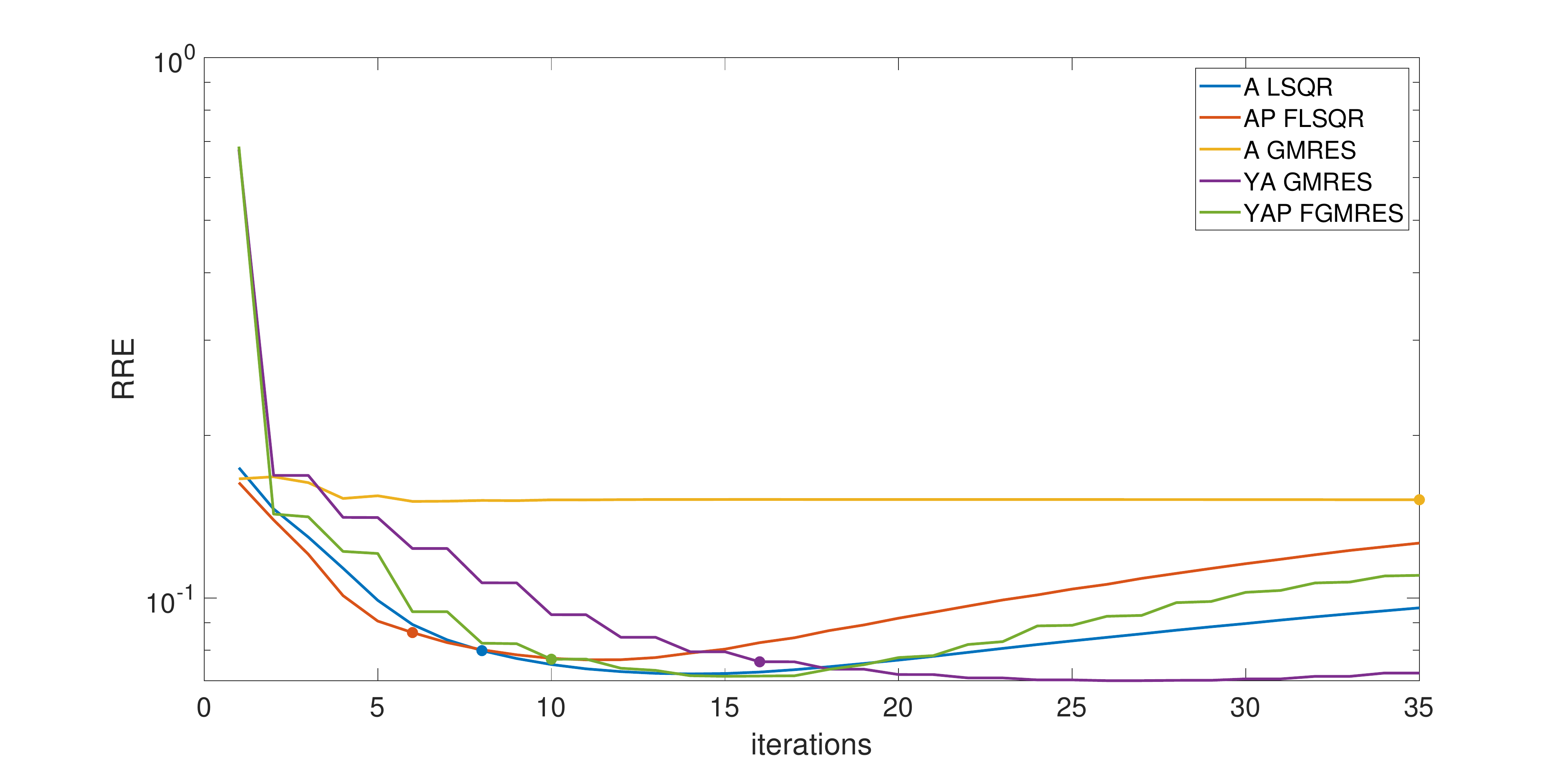}
		\caption{{Cameraman} test problem. Comparison between the error behaviours of the GMRES and FGMRES applied to the symmetrized system \eqref{eq:Ysys} and of the LSQR, FLSQR, and GMRES applied to the non-symmetrized linear system. The dots mark the iterations satisfying the discrepancy principle stopping criterion.}
		\label{fig:errors_camera_prec}
	\end{figure}
	
	\begin{table}
	\centering
		\begin{tabular}{c|ccc|cccc}
		& \multicolumn{3}{c}{\it Best Reconstruction} & \multicolumn{4}{c}{\it Discrepancy Principle}\\
		\hline
		Method & RRE & PSNR & iter & RRE & PSNR & iter & time (s)\\
			\hline
			A LSQR     & 0.0723 & 28.3279 & 14 & 0.0799  & 27.4657  &  8 &  1.3023\\
			AP FLSQR   & 0.0768 & 27.8058 & 11 & 0.0863  & 26.7965  &  6 &  1.3191\\
			A GMRES    & 0.1509 & 21.9410 &  6 & - & - & - & - \\           
			YA GMRES   & 0.0703 & 28.5773 & 26 & 0.0762  & 27.8791  & 16 &  1.4398\\
			YAP FGMRES & 0.0716 & 28.4157 & 15 & 0.0770  & 27.7841  & 10 &  1.2180\\
		\end{tabular}
		\caption{Cameraman test problem. RRE and PNSR values with the corresponding iteration numbers and computational times for the restoration with minimum RRE and for the restoration determined when terminating the iterations with the discrepancy principle.}\label{tab:times_psnr_camera}
	\end{table}
	\begin{figure}
		\includegraphics[width=0.3\textwidth]{Figures/camera_image_true.pdf}
		\includegraphics[width=0.3\textwidth]{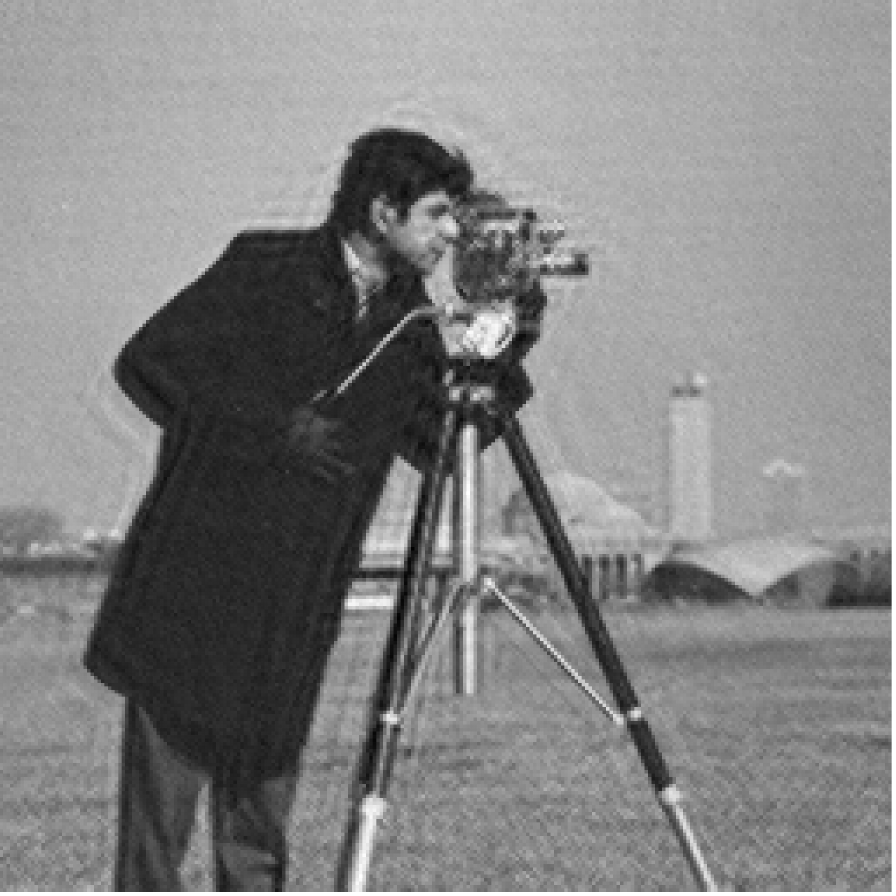}
		\includegraphics[width=0.3\textwidth]{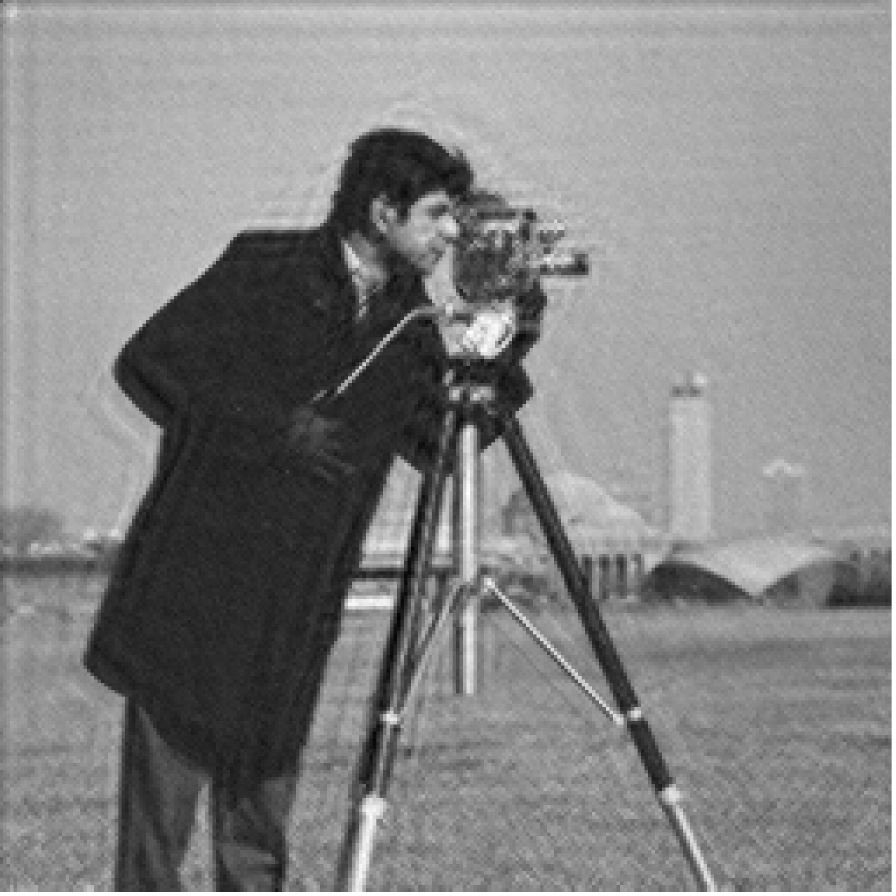}
		\caption{Cameraman test problem. Exact image (left), preconditioned GMRES reconstruction (center), preconditioned LSQR reconstruction (right) given by the discrepancy principle.}
		\label{fig:camera_images_reconstructed}
	\end{figure}

%------------------------------------------------------------------------------------------------------------------------------
		\subsection{Test Problem: Edges with reflective boundary conditions}\label{ssec:edges}
	In this case, we analyse the reconstruction of a sparse image containing the edges of geometrical figures contaminated by a severe shake blur and a noise of level $\sigma=0.1$. Figure \ref{fig:edges_images} shows the exact image, the PSF, and the blurred image, all of size $128\times 128$ pixels, in a colormap that better emphasises the light intensity of the grayscale images.
		\begin{figure}
		\includegraphics[width=0.3\textwidth]{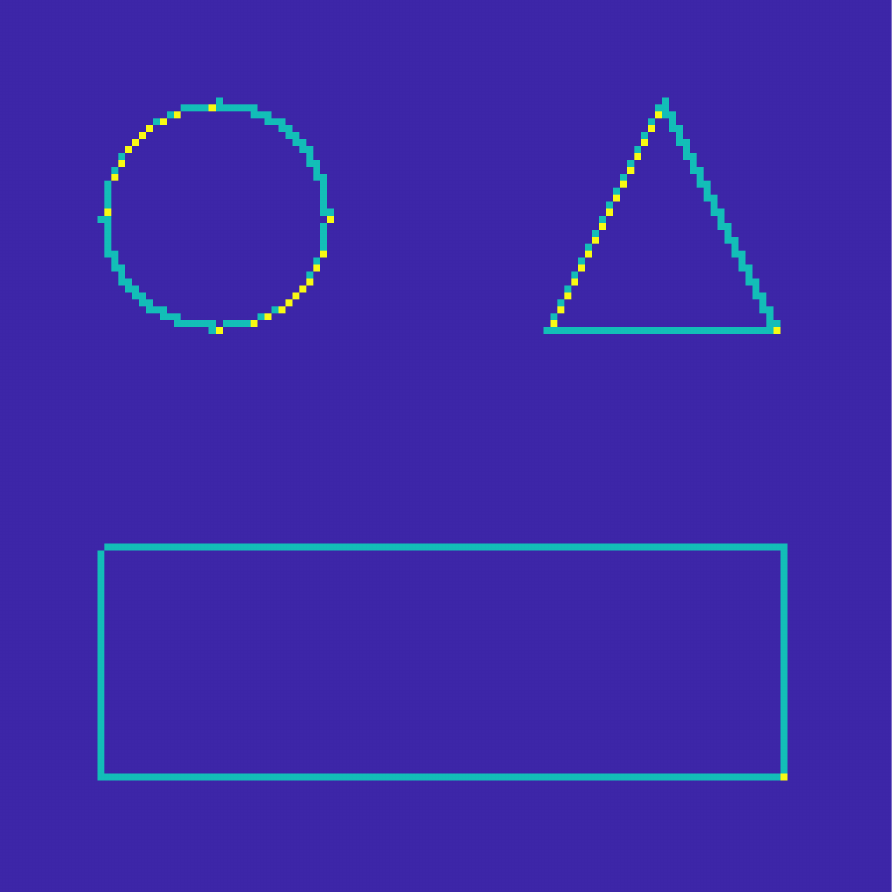}
		\includegraphics[width=0.3\textwidth]{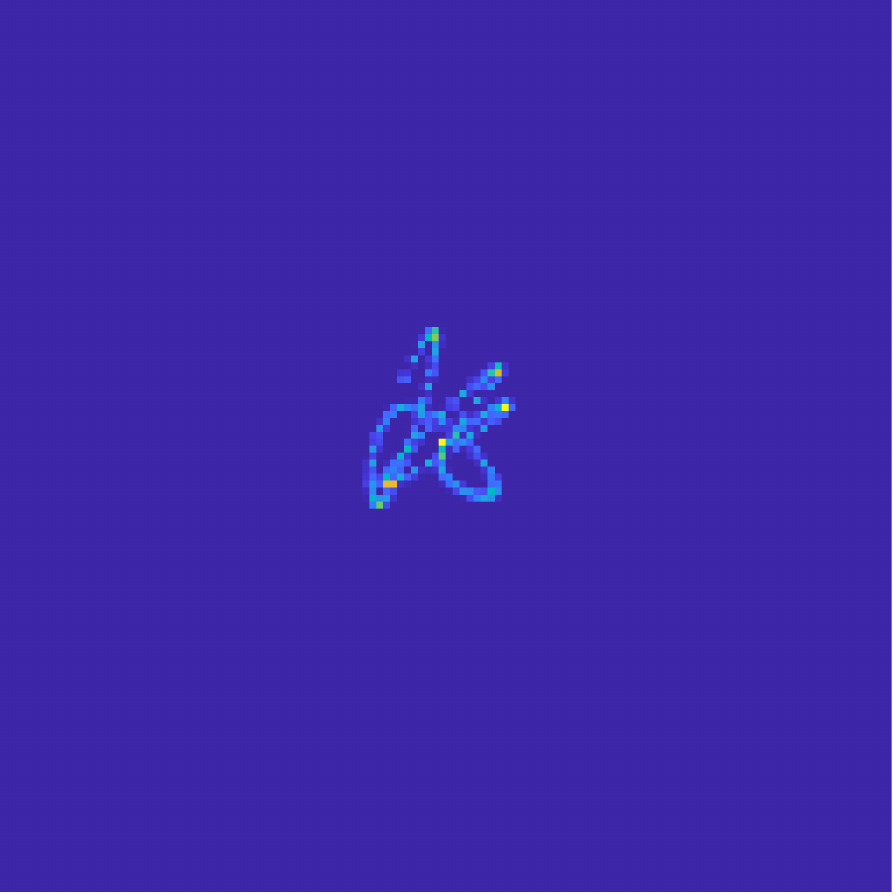}
		\includegraphics[width=0.3\textwidth]{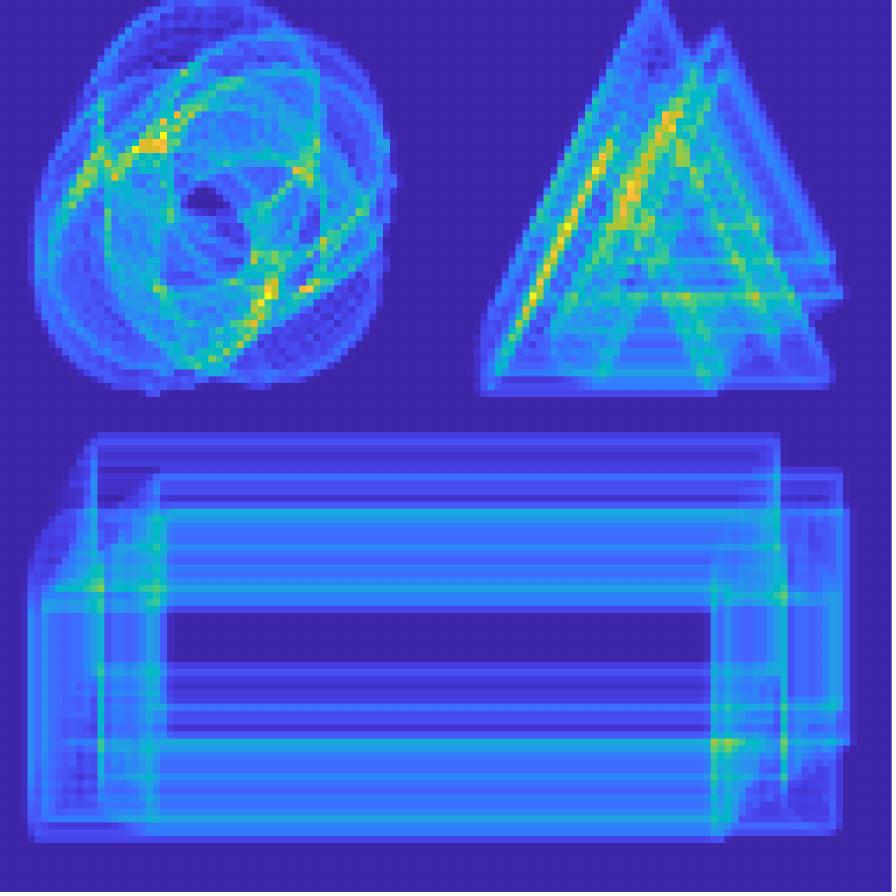}
		\caption{Edges test problem. From left to right: exact image, PSF, and blurred image.}
		\label{fig:edges_images}
	\end{figure}
	
	In Figure \ref{fig:errors_edges_prec}, we report the error behaviour of FGMRES applied to the symmetrized system \eqref{eq:Ysys} and of FLSQR and FGMRES applied to the original linear system \eqref{original}. We consider two different iteration-dependent preconditioners for FGMRES and FLSQR. The $W$ preconditioner implements the re-weighting strategy from \cite{MR4331956} used to enforce sparsity in the solution, as explained in Section \ref{sec:itReg}. Then, we combine the latter strategy with the circulant preconditioning technique of $\cir_n(|p_{\alpha_k}|)$ and $\cir_n(p_{\alpha_k})$, with $\alpha_k=0.1\cdot(0.8)^k$. To our knowledge, this is the first time that the combination of these two preconditioners has been explored with FLSQR. Between the two options of preconditioners $WP$ and $PW$, we choose $WP$ since it is the choice that enforces sparsity within the basis vectors of the approximation subspace for the solution. In support of this statement, in Figure \ref{fig:edges_krylov_basis} we show the first four, the eighth and the twelfth basis vectors of the flexible Krylov subspaces generated when applying FGMRES with the two different preconditioning strategies.
 
    Table \ref{tab:times_psnr_edges} shows the RRE and PNSR values with the corresponding iteration numbers and computational times for the restoration with minimum RRE and for the restoration determined when terminating the iterations with the discrepancy principle.
	
	Firstly, this test highlights that the sparse preconditioning technique from \cite{MR4331956}, which fails in combination with FGMRES when the matrix $A$ is highly non-symmetric as in this case, is instead a valid choice when applied to the matrix $YA$: the $YAW$ FGMRES is slightly better than the $AW$ LSQR method in terms of RRE and PSNR of the discrepancy principle reconstruction and it reaches the stopping criterion in comparable computational times. Moreover, regarding the combination with the circulant preconditioner, we see that with this choice of $\alpha_k$ the semi-convergence of both FGMRES and FLSQR is accelerated, but what is most remarkable is that for FGMRES the convergence becomes also more stable.
	
	Finally, we see in Figure \ref{fig:edges_images_reconstructed} that the reconstructions of $AW$ LSQR and $YAW$ FGMRES are comparable from a viewer's perspective.
	\begin{figure}
		\includegraphics[width=0.9\textwidth]{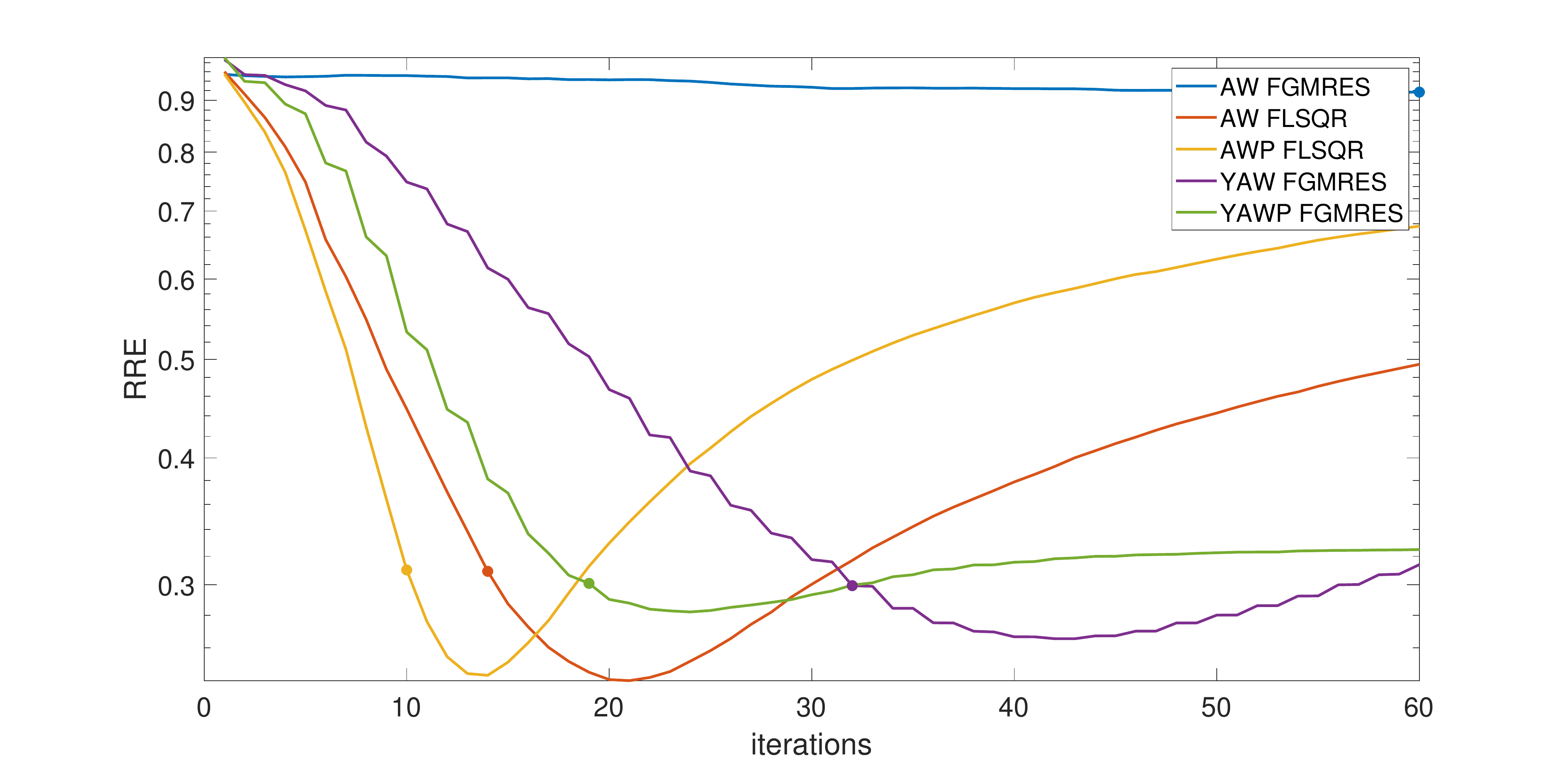}
		\caption{{Edges} test problem. Comparison between the error behaviours of FGMRES applied to the symmetrized system \eqref{eq:Ysys} and of FLSQR and FGMRES applied to the non-symmetrized linear system. The dots mark the iterations satisfying the discrepancy principle stopping criterion.}
		\label{fig:errors_edges_prec}
	\end{figure}
	
	\begin{table}
	\centering
		\begin{tabular}{c|ccc|cccc}
		& \multicolumn{3}{c}{\it Best Reconstruction} & \multicolumn{4}{c}{\it Discrepancy Principle}\\
		\hline
		Method & RRE & PSNR & iter & RRE & PSNR & iter & time (s)\\
			\hline
			AW FLSQR     & 0.2414 & 32.7000 & 21 & 0.3093 & 30.5465 & 14 &  0.8622 \\
			APW FLSQR    & 0.2443 & 32.5954 & 14 & 0.3103 & 30.5201 & 10 &  0.8258\\
			AW FGMRES    & 0.9166 & 21.1105 & 60 & - & - & - & - \\   
			YAW FGMRES   & 0.2655 & 31.8737 & 42 & 0.2994 & 30.8293 & 32 &  0.9335 \\
			YAPW FGMRES  & 0.2821 & 31.3466 & 24 & 0.3010 & 30.7835 & 19 &  0.8536 \\
		\end{tabular}
		\caption{Edges test problem. RRE and PNSR values with the corresponding iteration numbers and computational times for the restoration with minimum RRE and for the restoration determined when terminating the iterations with the discrepancy principle.}\label{tab:times_psnr_edges}
	\end{table}
 	\begin{figure}
            \begin{center}
		\includegraphics[width=\textwidth]{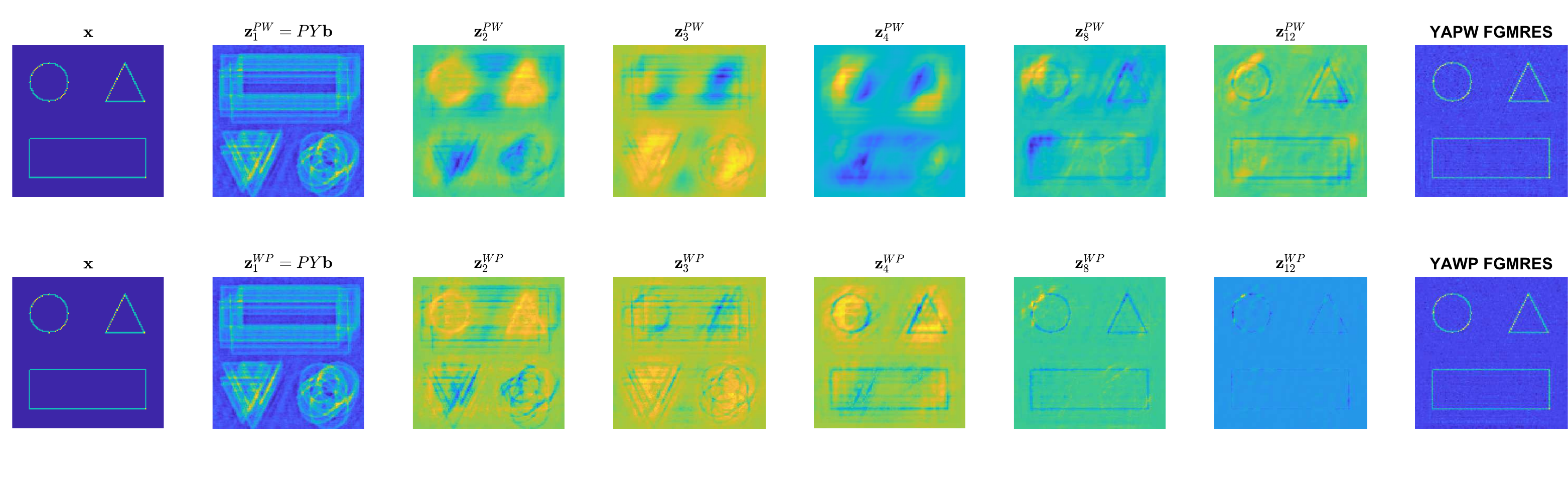}
		\caption{Edges test problem. \textbf{Row 1:} the true image $\mathbf{x}$, the basis images number 1, 2, 3, 4, 8, and 12 of the flexible Krylov subspace, and the reconstruction returned by the discrepancy principle when applying FGMRES to with a $PW$ iteration-dependent right preconditioning strategy. \textbf{Row 2:} the true image $\mathbf{x}$, the basis images number 1, 2, 3, 4, 8, and 12 of the flexible Krylov subspace, and the reconstruction returned by the discrepancy principle when applying FGMRES with a $WP$ iteration-dependent right preconditioning strategy.}
		\label{fig:edges_krylov_basis}
            \end{center}
	\end{figure}
	\begin{figure}
		\includegraphics[width=0.3\textwidth]{Figures/edges_image_true.pdf}
		\includegraphics[width=0.3\textwidth]{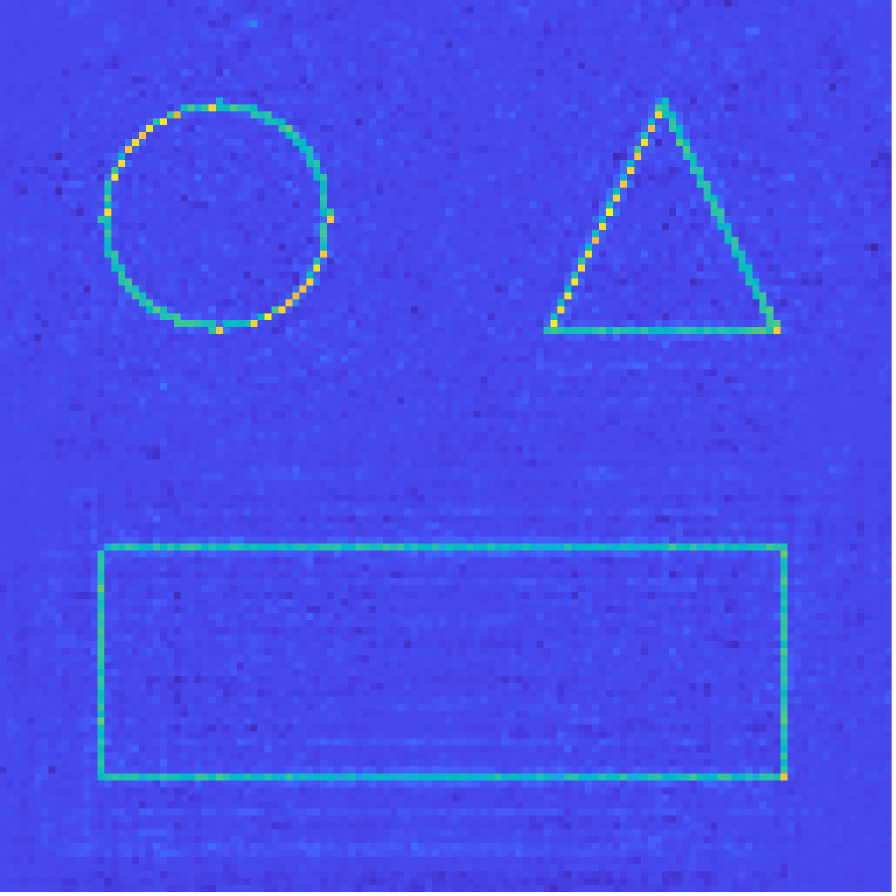}
		\includegraphics[width=0.3\textwidth]{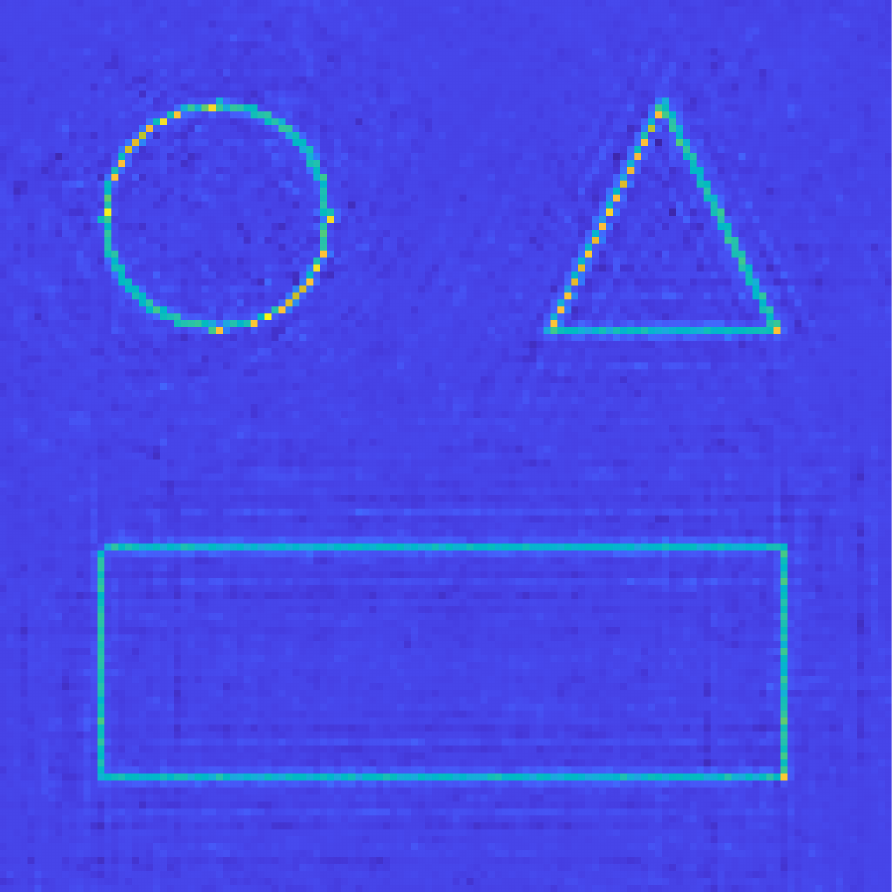}
		\caption{Edges test problem. Exact image (left), $YAW$ FGMRES discrepancy principle reconstruction (center), and $AW$ FLSQR discrepancy principle reconstruction (right).}
		\label{fig:edges_images_reconstructed}
	\end{figure}

\section{Conclusions and future work}\label{sec:end} This paper described a couple of preconditioning strategies to be used when solving image deblurring problems through iterative regularization methods. Some theoretical insight and substantial numerical experimentation showed that using such preconditioners within the MINRES, GMRES, FGMRES and FLSQR methods and with a variety of PSFs and boundary conditions leads to better reconstructions, which are computed more efficiently and often in a more stable way with respect to their unpreconditioned counterparts. 

There are a number of extensions to the present work. First of all, as hinted in Section \ref{sec:itReg}, one may consider different adaptive ways of setting the regularization parameter $\alpha$ within the regularizing circulant preconditioner $\cir_n(|p_\alpha|)$, including strategies based on the discrepancy principle and generalized cross validation. Secondly, although this paper focused on purely iterative regularization methods, all the Krylov subspace methods considered here can be used in a hybrid fashion, i.e., combining projection onto Krylov subspaces and Tikhonov regularization; see \cite{newsurvey}. An interesting new research avenue may explore the performance of the preconditioners considered here in a hybrid setting. Finally, when employing Krylov methods for the computation of sparse solutions, there exists popular alternatives to the flexible methods considered here, which are based on generalized Krylov subspaces; see \cite{Lothar1, Lothar2}. Future work can focus on the numerical study and analysis of preconditioning techniques based on regularizing circulant preconditioners applied to such methods.

\appendix
\section{Preconditioner Analysis Proof}\label{appendixGLT}
The proof of Theorem \ref{thm:preconditioner_cluster} is based on a distribution result in \cite{MR4304085}, which in turn exploits the concept of Generalized Locally Toeplitz (GLT) sequences.  The formal definition of the GLT class and the derivation of their properties require rather technical tools. We refer the reader to \cite{mbg} for a discussion on the GLT theory in its general multilevel block form. 

The crucial feature of a GLT sequence $\{A_n\}_n$ is that we can associate it to a function {$\tilde{f}:[0,1]^k\times [-\pi,\pi]^k \rightarrow \mathbb{C}^{s \times s}$, called GLT symbol. We denote this relation with $\{A_n\}_n \sim_{\textsc{glt}} \tilde{f}$. If all the matrices of the sequence are Hermitian, then $\tilde{f}$ is the eigenvalue symbol of $\{A_n\}_n$ in a sense analogous to formula \eqref{eq:spectral_distribution}. Every $k$-level Toeplitz sequence generated by a function $f \in L^1([-\pi,\pi]^k)$ is a GLT sequence and its symbol is {$\tilde{f}(\cdot,\boldsymbol{\vartheta}) = f(\boldsymbol{\vartheta})$}. For image deblurring problems $k=2$ and a 2-level Toeplitz matrix is a BTTB matrix.

GLT sequences constitute a $*$-algebra of matrix-sequences to which multilevel Toeplitz matrix-sequences with Lebesgue integrable generating functions belong. The sequence obtained via algebraic operations on a finite set of given GLT sequences is still a GLT sequence and its symbol is obtained by performing the same algebraic manipulations on the corresponding symbols of the input GLT sequences.

For the proof of the next theorem, we use the results in \cite{MR4304085}, where the authors make use of the GLT theory to discover the spectral distribution of $\{Y\toep_n(g)\}_{n}$ and of the preconditioned sequence $\{P_n^{-1}Y\toep_n(g)\}_{n}$.
\begin{theorem}
		Let $\varepsilon$ and $\tau$ be positive values such that $\varepsilon\in(0,1)$ and $\tau\in (0,\pi)$. Let $f\in L^1([-\pi,\pi]^2)$  be a bivariate function with real Fourier coefficients,  periodically extended to the whole real plane, and such that
		\begin{equation}
			\begin{cases}
				|f(\vartheta_1,\vartheta_2)| > {\varepsilon}, &  \mbox{if }\left|\vartheta_1^2+\vartheta_2^2\right|<\tau,\\
				|f(\vartheta_1,\vartheta_2)| \le \varepsilon, 		& \mbox{otherwise},
			\end{cases}
		\end{equation}
		and define 
		\begin{equation}
			g_\tau(\vartheta_1,\vartheta_2) = 
			\begin{cases}
				|{f(\vartheta_1,\vartheta_2)}|, &  \mbox{if }\left|\vartheta_1^2+\vartheta_2^2\right|<\tau,\\
				1, 		& \mbox{otherwise}.
			\end{cases}
		\end{equation}
%		Let ${C}_\mathbf{n}$ be a bi-level circulant matrix such that $\left\{{C}_\mathbf{n}\right\}_\mathbf{n} \sim_{\mbox{GLT}}  g$. Then
%		\begin{equation}
%			\left\{{C}_\mathbf{n}^{-1}YT_\mathbf{n}(f)\right\}_\mathbf{n} \sim_\lambda  \psi
%		\end{equation}
Then
		\begin{equation}\label{eqn:prec_distribution_app}
			\left\{\cir_n(g_\tau)^{-1}Y\toep_n(f)\right\}_n \sim_\lambda  \psi,
		\end{equation}
where
		\begin{equation}\label{eqn:prec_symbol_app}
			\psi(\vartheta_1,\vartheta_2) = 
			\begin{cases}
	            \frac{1}{|f|}\begin{bmatrix} 0 & f(\vartheta_1,\vartheta_2) \\ \\ \overline{f(\vartheta_1,\vartheta_2)} & 0 \end{bmatrix}, &  \mbox{if }\left|\vartheta_1^2+\vartheta_2^2\right|<\tau,\\ \\
				\begin{bmatrix} 0 & f(\vartheta_1,\vartheta_2) \\ \\ \overline{f(\vartheta_1,\vartheta_2)} & 0 \end{bmatrix}, & \mbox{otherwise}.
			\end{cases}
		\end{equation}
	\end{theorem}
	\begin{proof}
		Since $g_\tau$ is a strictly positive function, the matrices $\cir_n(g_\tau)$ are Hermitian positive definite for all $n$.
		According to \cite[Theorem 3.3]{MR4304085}, the spectral distribution \eqref{eqn:prec_distribution_app} holds under the assumptions
		\begin{enumerate}
		 	\item $\{\cir_n(g_\tau)\}_{n} \sim_{\textsc{glt}} g_\tau$;
		 	\item $\{\Pi_n U_n \cir_n(g_\tau) U_n \Pi_n^T\}_{n} \sim_{\textsc{glt}} g_\tau I_2$.
		\end{enumerate}
		where $\Pi_n$ and $U_n$ are the permutation matrices defined in \cite{MR4304085}.
%		\[
%			U_n = \begin{bmatrix} Y_{\lceil{\frac{n}{2}}\rceil} & \\ & I_{\lfloor{\frac{n}{2}}\rfloor}\end{bmatrix}
%		\]
		In order to prove 1. and 2., we state the following results:
		\begin{itemize}
			\item $\{\toep_n(g_\tau)\}_{n} \sim_{\textsc{glt}} g_\tau$, which is a GLT property;
			\item $\{\Pi_n U_n \toep_n(g_\tau) U_n \Pi_n^T\}_{n} \sim_{\textsc{glt}} g_\tau I_2$, which holds thanks to \cite[Remark 4]{MR4304085} and the fact that $g_\tau$ is real-valued;
			\item $\{\cir_n(g_\tau)-\toep_n(g_\tau)\}_{n} \sim_{\textsc{glt}} 0$; see Remark \ref{rmk:circulant_preconditioners}.
		\end{itemize}
		Combining the latter statements with the $*$-algebra property of GLT sequences, we deduce that 1. and 2. hold.
		Hence, $\left\{\cir_n(g_\tau)^{-1}Y\toep_n(f)\right\}_n \sim_\lambda  \psi$ and the proof is complete.
	\end{proof}
	\begin{remark}\label{rmk:circulant_preconditioners}
		For brevity's sake, in Section \ref{sec:structured_matrices} we stated that $\cir_n(f)$ is the Strang preconditioner of $\toep_n(f)$. This is true if $f$ is a trigonometric polynomial for $n$ large enough, which is the case of the functions $f$ related to the blurring operators that we are considering. The Strang preconditioner might not be a feasible choice for general functions $f\in L^1$; see \cite{MR2374567} for details. However, the Strang preconditioner that we consider for $g_\tau$ is related to the Strang preconditioner constructed from trigonometric polynomial. Indeed, this is the procedure to construct the preconditioner:
		\begin{enumerate}
			\item from the trigonometric polynomial $f$ find the eigenvalues of the Strang preconditioner $\cir_n(f)$;
			\item set to 1 the eigenvalues that are equal to or less than $\varepsilon$;
			\item construct the circulant matrix with the eigenvalues computed in Step 2.
		\end{enumerate}
		So, the eigenvalues of the circulant preconditioner are actually a uniform sampling of the function $g_\tau$ and hence the distribution $\{\cir_n(g_\tau)-\toep_n(g_\tau)\}_{n} \sim_{\textsc{glt}} 0$ holds, where the the notation $\cir_n(g_\tau)$ is an abuse of notation since, in general, $\cir_n(g_\tau)$ is neither the Strang preconditioner nor the circulant matrix generated by $g_\tau$, but it well-approximates both in the case where the Fourier coefficients of $g_\tau$ decay rapidly.
	\end{remark}

\section{Acknowledgments}
The work of the first and second authors was partially supported by the Gruppo Nazionale per il Calcolo Scientifico (GNCS).

\bibliographystyle{siamplain}
\bibliography{references}
\end{document}